\documentclass[12pt]{article}

\input xy

\xyoption{all}

\usepackage{amssymb,amsbsy,amsthm,amsmath,graphicx,mathrsfs,epsfig}

\usepackage{times}

\usepackage{mathabx}

\usepackage{sectsty}

\sectionfont{\large\center}

\subsectionfont{\normalsize}

\newtheorem{thm}{Theorem}[section]
\newtheorem*{thm*}{Theorem}
\newtheorem{lem}[thm]{Lemma}
\newtheorem{prop}[thm]{Proposition}
\newtheorem{cor}[thm]{Corollary}

\newtheorem{dfn}[thm]{Definition}

\newtheorem{ques}[thm]{Question}
\theoremstyle{remark}

\newtheorem*{rmk}{Remark}

\newcommand{\bs}[1]{\boldsymbol{#1}}
\renewcommand{\bf}[1]{\mathbf{#1}}
\renewcommand{\rm}[1]{\mathrm{#1}}
\renewcommand{\cal}[1]{\mathcal{#1}}

\newcommand{\bbN}{\mathbb{N}}
\newcommand{\bbQ}{\mathbb{Q}}
\newcommand{\bbR}{\mathbb{R}}
\newcommand{\bbT}{\mathbb{T}}
\newcommand{\bbZ}{\mathbb{Z}}

\newcommand{\sfE}{\mathbb{E}}
\newcommand{\sfP}{\mathbb{P}}
\newcommand{\sfQ}{\mathbb{Q}}

\newcommand{\rmH}{\mathrm{H}}

\newcommand{\rmS}{\mathrm{S}}

\renewcommand{\d}{\mathrm{d}}
\newcommand{\rme}{\mathrm{e}}
\newcommand{\rmh}{\mathrm{h}}
\newcommand{\rmi}{\mathrm{i}}

\newcommand{\F}{\mathcal{F}}

\newcommand{\calM}{\mathcal{M}}
\newcommand{\calN}{\mathcal{N}}
\newcommand{\calO}{\mathcal{O}}

\newcommand{\U}{\mathcal{U}}
\newcommand{\X}{\mathcal{X}}
\newcommand{\Y}{\mathcal{Y}}
\newcommand{\Z}{\mathcal{Z}}

\newcommand{\G}{\Gamma}

\renewcommand{\O}{\Omega}
\renewcommand{\S}{\Sigma}

\renewcommand{\a}{\alpha}
\renewcommand{\b}{\beta}
\newcommand{\eps}{\varepsilon}
\newcommand{\g}{\gamma}
\renewcommand{\l}{\lambda}
\renewcommand{\o}{\omega}
\newcommand{\s}{\sigma}
\renewcommand{\phi}{\varphi}
\renewcommand{\k}{\kappa}

\newcommand{\ol}[1]{\overline{#1}}

\newcommand{\fin}{\nolinebreak\hspace{\stretch{1}}$\lhd$}

\newcommand{\actson}{\curvearrowright}
\renewcommand{\to}{\longrightarrow}
\newcommand{\onto}{\twoheadrightarrow}

\renewcommand{\t}{\widetilde}
\newcommand{\lws}{\stackrel{\rm{lw}\ast}{\to}}

\newcommand{\q}{\stackrel{\rm{q}}{\to}}
\newcommand{\dq}{\stackrel{\rm{dq}}{\to}}
\newcommand{\aL}{\stackrel{\rm{aL}}{\to}}
\renewcommand{\Pr}{\mathrm{Prob}}

\begin{document}


\title{\textbf{\Large{The geometry of model spaces for probability-preserving actions of sofic groups}}}
\author{Tim Austin}

\date{}

\maketitle

\begin{abstract}
Bowen's notion of sofic entropy is a powerful invariant for classifying probability-preserving actions of sofic groups. It can be defined in terms of the covering numbers of certain metric spaces associated to such an action, the `model spaces'.

The metric geometry of these model spaces can exhibit various interesting features, some of which provide other invariants of the action.  This paper explores an approximate connectedness property of the model spaces, and uses it give a new proof that certain groups admit factors of Bernoulli shifts which are not Bernoulli.  This was originally proved by Popa.  Our proof covers fewer examples than his, but provides additional information about this phenomenon.
\end{abstract}

\noindent MSC(2010): 37A35 (primary), 28D20, 37A25, 37A50, 60J10 (secondary)

\vspace{7pt}

\noindent Kewords: Sofic entropy, Bernoulli system, Popa factor, model spaces, connectedness

\setcounter{tocdepth}{1}
\tableofcontents

\section{Introduction}

Let $G$ be a discrete sofic group, $(X,\mu)$ a standard probability space, and ${T:G\actson X}$ a measurable action preserving $\mu$.  The triple $(X,\mu,T)$ is called a \textbf{$G$-system} or just a \textbf{system}. Fix a choice of sofic approximation $\S = (\s_n)_{n\geq 1}$ to $G$, where each $\s_n$ is a map from $G$ to $\rm{Sym}(V_n)$ for some finite set $V_n$.

The sofic entropy of $(X,\mu,T)$ relative to this sofic approximation gives a generalization of the classical Kolmogorov-Sinai entropy which is still available if $G$ is not amenable.  If $(X,\mu,T)$ has a finite generating partition, then its sofic entropy rel $\S$ was defined by Lewis Bowen in~\cite{Bowen10}, and this definition was generalized to any $G$-system by Kerr and Li in~\cite{KerLi11b,KerLi13}.  In case $G$ is amenable, the agreement between sofic entropy rel $\S$ and Kolmogorov-Sinai entropy was shown in~\cite{Bowen12}.

The present paper introduces a new invariant of $G$-systems, constructed from similar ingredients to the sofic entropy.  Our starting point is the formulation of sofic entropy given in~\cite{Aus--soficentadd}, which slightly modifies~\cite{KerLi13} without changing the resulting entropy values.

In the first place, we consider a special class of $G$-systems.  A \textbf{metric $G$-process} is a quadruple $(\X^G,\mu,S,d)$ in which $\X$ is another standard measurable space, $d$ is a compact metric on $\X$ whose Borel sets are the $\s$-algebra of $\X$, $S$ is the right-shift action on $\X^G$, and $\mu$ is an $S$-invariant probability.  Then $\rmh_\S(\mu)$ is defined in terms of the covering numbers of certain metric spaces of `good models' for $\mu$ over $\S$.  The `good models' used in~\cite{Aus--soficentadd} differ slightly from their earlier counterparts in~\cite{KerLi11b,KerLi13}, but it is shown in~\cite[Subsection 3.2]{Aus--soficentadd} that one always obtains the same values for sofic entropy as in those earlier papers.  The resulting entropy value $\rmh_\S(\mu)$ does not depend on the choice of $d$ and is an isomorphism-invariant of the $G$-system $(\X^G,\mu,S)$.  Since any $G$-system is isomorphic to a shift-system, this lets us extend the definition of $\rmh_\S$ unambiguously to arbitrary $G$-systems.

In the approach of~\cite{Aus--soficentadd}, and similarly to the other approaches, `good models' for $(\X^G,\mu,S,d)$ are certain elements of the finite-dimensional Cartesian powers $\X^{V_n}$.  Each $\bf{x} \in \X^{V_n}$ has an `empirical distribution' $P^{\s_n}_\bf{x}$ which may be used to describe its `local statistics' over certain subsets of $V_n$: see Subsection~\ref{subs:AL}.  It is a `good model' if its empirical distribution is close to $\mu$ in the weak$^\ast$ topology on $\Pr(\X^G)$.  For any weak$^\ast$-neighbourhood $\calO$ of $\mu$, let
\[\O(\calO,\s_n) := \{\bf{x} \in \X^{V_n}:\ P^{\s_n}_\bf{x} \in \calO\}\]
be the set of models which are `good' according to this neighbourhood $\calO$.

Also, for each $n$, let $d^{(V_n)}$ be the Hamming average of copies of $d$ on $\X^{V_n}$:
\[d^{(V_n)}(\bf{x},\bf{x}') = \frac{1}{|V_n|}\sum_{v \in V_n}d(x_v,x'_v) \quad \hbox{for}\ \bf{x} = (x_v)_{v \in V_n},\ \bf{x}' = (x'_v)_{v \in V_n}.\]
Sofic entropy is defined in terms of certain asymptotic geometric features of the sequences of metric spaces
\[\big(\O(\calO,\s_n),d^{(V_n)}\big).\]
Specifically, it is the quantity
\[\rmh_\S(\mu) := \sup_{\delta > 0}\ \inf_{\calO}\ \limsup_{n\to\infty}\frac{1}{|V_n|}\log \rm{cov}_\delta\big(\O(\calO,\s_n),d^{(V_n)}\big),\]
where $\rm{cov}_\delta$ is the $\delta$-covering number of the given metric space and $\calO$ runs over all weak$^\ast$-neighbourhoods of $\mu$.

The original motivation for studying sofic entropy was the classification of Bernoulli systems: see~\cite{Bowen10,KerLi11a}.  It has quickly found numerous other applications, and suggests many directions for further research.  One of these is to look for other properties of the metric spaces $(\O(\calO,\s_n),d^{(V_n)})$ that carry some dynamical information about $\mu$, and might give rise to other invariants.

The present paper focuses on one such property, which we refer to as `connected model spaces rel $\S$'.  It is an approximate kind of connectedness for the spaces $\O(\calO,\s_n)$.  In case $\X$ is finite, these metric spaces are discrete, so we do not ask for their connectedness as in classical point-set topology. Rather, we consider whether connectedness holds in a suitable asymptotic sense as $\calO$ becomes smaller and $n\to\infty$.  The precise notion we need is given in Definition~\ref{dfn:connected-model}.

Once that definition has been made, we can show that the property of connected model spaces rel $\S$ is an isomorphism-invariant.  In fact we prove more than this.  Some factor maps do not preserve the property of connected model spaces, but we introduce a class of factor maps which do, and show that they include all isomorphisms.

Given a factor map
\[\Phi:(\X^G,\mu,S) \to (\Y^G,\nu,S),\]
and also compact metrics $d_\X$ on $\X$ and $d_\Y$ on $\Y$ which generate their $\s$-algebras, the developments in~\cite{Aus--soficentadd} include the construction of suitable `approximating maps' from $\X^{V_n}$ to $\Y^{V_n}$.  These send good models for $\mu$ to good models for $\nu$, up to various errors that must be carefully controlled in terms of the metrics $d_\X^{(V_n)}$ and $d_\Y^{(V_n)}$.  Put roughly, we say that $\Phi$ is `model-surjective rel $\S$' if \emph{all} good models for $(\Y^G,\nu,S)$ are close to images of good models for $(\X^G,\mu,S)$ under these approximating maps.  This is made precise in Definition~\ref{dfn:model-surj}.

Having introduced this special property of factor maps, we will show that it is preserved by composition, and that all isomorphisms have this property.  Then we prove the following.

\vspace{7pt}

\noindent\textbf{Theorem A}\quad \emph{Let $\Phi$ be a factor map as above. If $(\X^G,\mu,S,d_\X)$ has connected model spaces rel $\S$ and $\Phi$ is model-surjective rel $\S$, then $(\Y^G,\nu,S,d_\Y)$ also has connected model spaces rel $\S$.}

\emph{In particular, having connected model spaces rel $\S$ is a property only of the shift-system $(\X^G,\mu,S)$, not depending on the choice of $d_\X$, and is an isomorphism-invariant.}

\vspace{7pt}

Having proved Theorem A, the property of having connected model spaces rel $\S$ can be extended unambiguously to any $G$-system, by picking an isomorphism from it to $G$-process.

Although factor maps need not preserve the connectedness of model spaces rel $\S$, it turns out inverse limits do.

\vspace{7pt}

\noindent\textbf{Theorem B}\quad \emph{If all members of an inverse sequence of $G$-systems have connected model spaces rel $\S$, then so does the inverse limit.}

\vspace{7pt}

The rest of the paper is given to some examples of connected and non-connected model spaces.  In the first place, we have the following.

\vspace{7pt}

\noindent\textbf{Theorem C}\quad \emph{For any sofic group $G$ and sofic approximation $\S$, Bernoulli systems over $G$ have connected model spaces rel $\S$.}

\vspace{7pt}

On the other hand, for some groups $G$ and sofic approximations $\S$ there are factors of Bernoulli shifts that do not have connected model spaces (and so, in particular, not all factor maps are model-surjective).  The last section of the paper is given to a family of such examples.

Let $m$ be the Haar probability measure on the circle $\bbT$, and consider the factor map of $G$-systems
\begin{eqnarray}\label{eq:Popa}
(\bbT^G,m^{\times G},S)\to (X,\mu,T)
\end{eqnarray}
defined by forming the quotient by the diagonal subgroup of $\bbT^G$: that is,
\[X = \bbT^G/\{(\ldots,\theta,\theta,\ldots):\,\theta \in \bbT\},\]
$\mu$ is the Haar measure on this quotient group, and $T$ is the quotient action.

\vspace{7pt}

\noindent\textbf{Theorem D}\quad  \emph{If $G$ is a residually finite group with Kazhdan's property (T), then $G$ has a sofic approximation $\S$ relative to which the factor system $(X,\mu,T)$ constructed above does not have connected model spaces.}

\vspace{7pt}

In particular, this implies that $(X,\mu,T)$ is not Bernoulli.

These examples form a special case of a construction of non-Bernoulli factors of Bernoulli shifts due to Popa~\cite{Pop06b}, so for general $G$ we refer to $(X,\mu,T)$ as the \textbf{Popa factor} of $(\bbT^G,m^{\times G},S)$.  These factors show a striking difference between ergodic theory for amenable and non-amenable groups: among the former, factors of Bernoulli shifts are always still Bernoulli~\cite{OrnWei87}.  Popa's proof in~\cite{Pop06b} gives examples for a considerably larger range of groups $G$ than those in Theorem D.  It rests on calculations of the cohomology of these actions, generalizing work of Popa and Sasyk~\cite{PopSas07}.  Our proof of Theorem D also uses some (more elementary) cohomology theory, and it covers fewer examples, but it gives a geometric interpretation to the non-Bernoullicity of these factors.

As pointed out to me by Brandon Seward, combining Theorems B, C and D immediately gives the following.

\vspace{7pt}

\noindent\textbf{Corollary D$^\prime$}\quad \emph{If $G$ is as in Theorem D, then there are factors of Bernoulli $G$-systems which are not inverse limits of Bernoulli $G$-systems. \qed}

\vspace{7pt}

This contrasts with actions of amenable groups, among which inverse limits of Bernoulli systems, like factors, are also always still Bernoulli~\cite{OrnWei87}.

In the reverse direction, there are many groups which admit inverse limits of Bernoulli shifts which are not factors of Bernoulli shifts.  Indeed, if $G$ has a non-amenable free subgroup, then Bowen has shown~\cite{Bowen11weakiso} that every Bernoulli shift over $G$ factors onto every other Bernoulli shift (this property may in fact hold for all non-amenable $G$: see~\cite[Corollary 1.6]{Bowen12almostOrn} for partial progress).  One may therefore form an inverse sequence
\[\cdots \to (A_3^G,p_3^{\times G},S) \to (A_2^G,p_2^{\times G},S) \to (A_1^G,p_1^{\times G},S) \]
of Bernoulli systems in which each $A_n$ is a finite alphabet and the Shannon entropies $\rmH(p_n)$ have finite sum.  Now the inverse limit of this sequence has generating partitions of arbitrarily small Shannon entropy, by joining the generating partitions of all systems that are sufficiently high up the sequence.  As a result, that inverse limit has Rokhlin entropy zero, and hence also sofic entropy at most zero~\cite{Seward--KriI,Seward--KriII}.  On the other hand, every factor of a Bernoulli shift has positive sofic entropy~\cite{Ker13}.

Our tools can also be used to prove another strengthening of Theorem D.  Recall that a factor map $\Phi:(X,\mu,T)\to (Y,\nu,S)$ of $G$-systems is \textbf{complemented} if there is another factor map $\Psi:(X,\mu,T)\to (Z,\theta,R)$ such that the combined map $(\Phi,\Psi)$ is a measure-theoretic isomorphism
\[(X,\mu,T) \stackrel{\cong}{\to} (Y\times Z,\nu\times \theta,S\times R).\]
In this case $(Y,\nu,S)$ is a \textbf{complemented factor} of $(X,\mu,T)$.

The Popa factors give examples of factor maps of Bernoulli shifts that are not model-surjective.  However, it turns out that all complemented factor maps of Bernoulli shifts (and actually of a much more general class of systems) are model-surjective: see Theorem~\ref{thm:comp-fact-model-surj}.  Combined with Theorems A and C, this shows the following.

\vspace{7pt}

\noindent\textbf{Corollary D$^{\prime\prime}$}\quad \emph{If $G$ is as in Theorem D then the Popa factor $(X,\mu,T)$ is not a complemented factor of any Bernoulli shift.}

\vspace{7pt}

This corollary was also suggested to me by Brandon Seward.  It can also be proved using Popa's cohomological approach, as has been shown to me by Yongle Jiang.  The main new ingredient is the fact that, if $(X,\mu,T)$ is a complemented factor of any Bernoulli shift $(\X^G,\nu^{\times G},S)$, then the resulting inclusion homomorphism from the degree-one, $\bbT$-valued cohomology of the former into the latter is injective.  Similar arguments can be found in Jiang's recent paper~\cite{Jia}.

Finally, the above discussion suggests the following.

\begin{ques}
Is it true that for any group $G$, any complemented factor of a Bernoulli system is Bernoulli? \fin
\end{ques}

I do not know that the answer is Yes for any non-amenable group.

%
%
%
%
%
%
%
%
%
%

\section{Connected model spaces}

\subsection{Model spaces and maps between them}\label{subs:AL}

This part of the paper follows closely the approach to sofic entropy developed in~\cite[Part I]{Aus--soficentadd}, and mostly uses the same notation.

We use `big-$O$' and `little-$o$' notation without further comment. Among real numbers, we sometimes write `$a \approx_\eps b$' in place of `$|a-b| < \eps$'.

If $(X,d_X)$ and $(Y,d_Y)$ are two metric spaces, $\eps > 0$ and $L < \infty$, then a map $\phi:X \to Y$ is \textbf{$\eps$-almost $L$-Lipschitz} if
\[d_Y(\phi(x),\phi(x')) \leq \eps + Ld_X(x,x') \quad \forall x,x' \in X.\]
A map is \textbf{$\eps$-almost Lipschitz} if this holds for some $L$.

We next recall some of the basic definitions and results from~\cite[Part I]{Aus--soficentadd}.  They are mostly small modifications to~\cite{Bowen10,KerLi13}.  Suppose that $(\X^G,\mu,S,d)$ is a \textbf{metric $G$-process}, meaning that $(\X,d)$ is a compact metric space, $S$ is the right-shift action of $G$ on $\X^G$, and $\mu \in \Pr(\X^G)$ is $S$-invariant.

Given a finite set $V$ and a map $\s:G\to \rm{Sym}(V)$, and also $v \in V$ and $\bf{x} \in \X^V$, we define the \textbf{pullback name of $\bf{x}$ at $v$} by
\[\Pi^\s_v(\bf{x}) := (x_{\s^g(v)})_{g\in G} \in \X^G.\]
In terms of these, the \textbf{empirical distribution} of $\bf{x}$ is
\[P^\s_\bf{x} := \frac{1}{|V|}\sum_{v\in V}\delta_{\Pi^\s_v(\bf{x})}.\]
For any w$^\ast$-neighbourhood $\calO$ of $\mu$ in $\Pr(\X^G)$, the \textbf{$\calO$-good models for $\mu$} are the elements of
\[\O(\calO,\s_n) := \{\bf{x}:\ P^\s_\bf{x} \in \calO\}.\]
Motivation for these concepts and their use in defining sofic entropy can be found in~\cite[Subsection 3.1]{Aus--soficentadd}.

Given any map $\psi:\X^G\to \Y$ and also a map $\s:G\to \rm{Sym}(V)$ for some finite set $V$, we define the associated map $\psi^\s:\X^V\to \Y^V$ by
\[\psi^\s(\bf{x}) := \big(\psi(\Pi^\s_v(\bf{x}))\big)_{v\in V},\]
as in~\cite[Subsection 4.2]{Aus--soficentadd}.

Now let $\Phi = \phi^G:(\X^G,\mu,S,d_\X) \to (\Y^G,\nu,S,d_\Y)$ be a factor map of metric $G$-processes, where we use the same notation as in~\cite[Subsection 4.1]{Aus--soficentadd}.  As in that reference, an \textbf{$\eta$-almost Lipschitz} (or \textbf{$\eta$-AL}) \textbf{approximation to $\phi$ rel $(\mu,d_\X,d_\Y)$} is a measurable map $\psi:\X^G \to \Y$ with the following properties.
\begin{itemize}
\item[i)] The map $\psi$ approximates $\phi$ in the sense that
\begin{eqnarray}\label{eq:int-approx}
\int d_\Y(\phi(x),\psi(x))\,\mu(\d x) < \eta.
\end{eqnarray}
\item[ii)] There is a finite $D \subseteq G$ such that $\psi$ is $D$-local: that is, it depends only on coordinates in $D$.
\item[iii)] There is a $D$-local open subset $U \subseteq \X^G$ such that $\mu(U) > 1 -\eta$ and such that $\psi|U$ is $\eta$-almost Lipschitz from $d_\X^{(D)}$ to $d_\Y$.
\end{itemize}

These always exist for every $\eta$~\cite[Lemma 4.3]{Aus--soficentadd}.  In the applications below, we will more often consider sequences of such approximations: an \textbf{almost Lipschitz approximating sequence for $\phi$ rel $(\mu,d_\X,d_\Y)$} is a sequence of maps $\psi_k:\X^G\to \Y$ which are $\eta_k$-AL approximations to $\phi$ rel $(\mu,d_\X,d_\Y)$ for some positive parameters $\eta_k \to 0$.  This situation is denoted by $\psi_k \aL \phi$.

Although we will be using such maps to study isomorphism-invariant properties of systems, the definitions above depend crucially on the choice of the compact metrics $d_\X$ and $d_\Y$.

Later in this section, we will need some estimates from~\cite{Aus--soficentadd} concerning AL approximations and the corresponding maps on model spaces.  Suitable versions are recalled in the following lemmas.

\begin{lem}[Good models are mapped to good models:~{\cite[Proposition 4.10]{Aus--soficentadd}}]\label{lem:approx-by-Lip-maps}
For every w$^\ast$-neighbourhood $\calN$ of $\nu$ there is an $\eta > 0$ with the following property.  If $\psi$ is an $\eta$-AL approximation to $\phi$ rel $(\mu,d_\X,d_\Y)$, then there is a w$^\ast$-neighbourhood $\calO$ of $\mu$ such that
\[\psi^{\s_n}\big(\O(\calO,\s_n)\big) \subseteq \O(\calN,\s_n)\]
for all sufficiently large $n$. \qed
\end{lem}

\begin{lem}[Almost Lipschitz action on model spaces:~{\cite[Lemma 4.9]{Aus--soficentadd}}]\label{lem:Lip-still-Lip}
If $\psi$ is an $\eta$-AL approximation to $\phi$ rel $(\mu,d_\X,d_\Y)$, then there are $K < \infty$ and a w$^\ast$-neighbourhood $\calO$ of $\mu$ such that
\[\psi^{\s_n}|\O(\calO,\s_n)\]
is $(3\eta)$-almost $K$-Lipschitz from $d_\X^{(V_n)}$ to $d_\Y^{(V_n)}$ for all sufficiently large $n$. \qed
\end{lem}

For our later application of these results, it is convenient to combine them into a single assertion in terms of bases for the w$^\ast$-topologies around $\mu$ and $\nu$.  Since those topologies are metrizable, there exist such bases which are decreasing sequences of open neighbourhoods.  We will frequently use this fact without further explanation.

\begin{cor}\label{cor:combined-good-behav}
Suppose that $\psi_k \aL \phi$ rel $(\mu,d_\X,d_\Y)$.  Then there are
\begin{itemize}
\item parameters $\eps_k \downarrow 0$ and $K_k < \infty$,
\item a sequence of w$^\ast$-neighbourhoods $\calO_1 \supseteq \calO_2 \supseteq \dots$ of $\mu$,
\item and a base of w$^\ast$-neighbourhoods $\calN_1 \supseteq \calN_2 \supseteq \dots$ at $\nu$
\end{itemize}
such that for each $k$ both of the following hold for all sufficiently large $n$:
\[\psi_k^{\s_n}\big(\O(\calO_k,\s_n)\big) \subseteq \O(\calN_k,\s_n)\]
and
\[\psi_k^{\s_n}|\O(\calO_k,\s_n) \quad \hbox{is $\eps_k$-almost $K_k$-Lipschitz from $d_\X^{(V_n)}$ to $d_\Y^{(V_n)}$}.\]
\end{cor}

\begin{proof}
Let $\eta_k \downarrow 0$ be a sequence such that $\psi_k$ is an $\eta_k$-AL approximation to $\phi$ rel $\mu$ for every $k$, and set $\eps_k := 3\eta_k$ for each $k$.

Let $\calM_1 \supseteq \calM_2 \supseteq \dots$ be a base of w$^\ast$-neighbourhoods at $\nu$ with $\calM_1 = \Pr(\Y^G)$.  By Lemma~\ref{lem:approx-by-Lip-maps}, for each $j\in \bbN$ there are $k_0(j) \in \bbN$ and w$^\ast$-neighbourhoods $\calO'_{j,k}$ for each $k\geq k_0(j)$ such that for every $k\geq k_0(j)$ we have
\[\psi_k^{\s_n}\big(\O(\calO'_{j,k},\s_n)\big) \subseteq \O(\calM_j,\s_n) \quad \hbox{for all sufficiently large $n$}.\]
We may take $k_0(1) = 1$ because $\calM_1 = \Pr(\Y^G)$, and we may also assume that $k_0(j) \to \infty$ as $j\to\infty$. Now let
\[\calN_k := \calM_{\max\{j:\ k_0(j) \leq k\}}\]
and choose w$^\ast$-neighbourhoods $\calO'_k$ of $\mu$ satisfying
\[\calO'_k \subseteq \bigcap_{j:\ k_0(j) \leq k}\calO'_{j,k}.\]
By forming running intersections we may also assume that $\calO'_1 \supseteq \calO'_2 \supseteq \dots$.  For each $k$, letting $j$ be the largest integer for which $k_0(j) \leq k$, it follows that
\[\psi_k^{\s_n}\big(\O(\calO'_k,\s_n)\big) \subseteq \psi_k^{\s_n}\big(\O(\calO'_{j,k},\s_n)\big),\]
and this is contained in $\O(\calM_j,\s_n) = \O(\calN_k,\s_n)$ for all sufficiently large $n$.

On the other hand, Lemma~\ref{lem:Lip-still-Lip} gives values $K_k < \infty$ and w$^\ast$-neighbourhoods $\calO''_k$ of $\mu$ such that $\psi_k^{\s_n}|\O(\calO_k'',\s_n)$ is $\eps_k$-almost $K_k$-Lipschitz for all sufficiently large $n$, for every $k\geq 1$.  Once again we may assume that $\calO''_1 \supseteq \calO''_2 \supseteq \dots$.

Finally, let $\calO_k := \calO'_k \cap \calO''_k$ for each $k$.
\end{proof}

The deduction of this corollary from Lemmas~\ref{lem:approx-by-Lip-maps} and~\ref{lem:Lip-still-Lip} is similar to the deduction of `sequence versions' of those lemmas in~\cite[Subsection 4.3]{Aus--soficentadd}.  But it seems easier to make this deduction from scratch here, rather than adapting the corollaries in that subsection.

Corollary~\ref{cor:combined-good-behav} has a further consequence that is worth recording by itself.

\begin{cor}\label{cor:approx-by-Lip-maps}
If $\psi_k \aL \phi$ rel $(\mu,d_\X,d_\Y)$ and if $\bf{x}_n \in \X^{(V_n)}$ is a sequence of models such that $P^{\s_n}_{\bf{x}_n} \stackrel{\rm{weak}^\ast}{\to} \mu$, then
\[P^{\s_n}_{\psi^{\s_n}_{k_n}(\bf{x}_n)} \stackrel{\rm{weak}^\ast}{\to} \nu\]
whenever the sequence $k_1 \leq k_2 \leq \dots$ grows sufficiently slowly.
\end{cor}

\begin{proof}
This really uses only the first conclusion of Corollary~\ref{cor:combined-good-behav}: there are sequences $\calO_1 \supseteq \calO_2 \supseteq \dots$ and $\calN_1 \supseteq \calN_2 \supseteq \dots$ as in that corollary such that for all $k$ we have
\[\psi_k^{\s_n}\big(\O(\calO_k,\s_n)\big)\subseteq \O(\calN_k,\s_n) \quad \hbox{for all sufficiently large $n$}.\]

Since $P^{\s_n}_{\bf{x}_n} \stackrel{\rm{weak}^\ast}{\to} \mu$, for every $k$ there is an $N(k)$ such that $\bf{x}_n \in \O(\cal{O}_k,\s_n)$ for all $n\geq N(k)$.  Provided the sequence $k_1 \leq k_2 \leq \dots$ grows sufficiently slowly, it follows that $\bf{x}_n \in \O(\calO_{k_n},\s_n)$ for all sufficiently large $n$, and also that
\[\psi_{k_n}^{\s_n}\big(\O(\calO_{k_n},\s_n)\big)\subseteq \O(\calN_{k_n},\s_n) \quad \hbox{for all sufficiently large $n$}.\]
These together imply that
\[\psi^{\s_n}_{k_n}(\bf{x}_n) \in \O(\calN_{k_n},\s_n) \quad \hbox{for all sufficiently large $n$},\]
and so their empirical distributions converge to $\nu$, since $\calN_1 \supseteq \calN_2 \supseteq \dots$ is a base for the w$^\ast$-topology at $\nu$.
\end{proof}

\subsection{Connected model spaces}\label{subs:connected}

Let $(Y,d_Y)$ be a metric space, let $x,y \in Y$, and let $\delta > 0$.  A \textbf{$\delta$-path from $x$ to $y$} is a finite sequence
\[x = x_0,x_1,\ldots,x_\ell = y\]
in $Y$ such that $d_Y(x_i,x_{i+1}) < \delta$ for every $i \in \{0,\dots,\ell-1\}$.  The integer $\ell$ is the \textbf{length} of this $\delta$-path.  If $A \subseteq Y$ and $\delta > 0$, then $A$ is \textbf{$\delta$-connected} (\textbf{according to $d_Y$}) if for any $x,y \in A$ there is a $\delta$-path from $x$ to $y$ contained in $A$.

We are now ready to define our new property of metric $G$-processes:

\begin{dfn}[Connected model spaces rel $\S$]\label{dfn:connected-model}
Let $(\X^G,\mu,S,d)$ be a metric $G$-process.  It has \textbf{connected model spaces rel $\S$} if the following holds:
\begin{quote}
If $n_1 < n_2 < \dots$, and $\bf{x}_i,\bf{y}_i \in \X^{V_{n_i}}$ are two sequences satisfying
\[P^{\s_{n_i}}_{\bf{x}_i},P^{\s_{n_i}}_{\bf{y}_i} \stackrel{\rm{weak}^\ast}{\to} \mu,\]
then there are a sequence $\delta_i \downarrow 0$ and a sequence of $\delta_i$-paths \[\{\bf{x}_i = \bf{x}_{i,0},\ \bf{x}_{i,1},\ \dots,\ \bf{x}_{i,\ell_i} = \bf{y}_i\} \subseteq \X^{V_{n_i}}\]
(according to the metrics $d^{(V_{n_i})}$) such that for any w$^\ast$-neighbourhood $\calO$ of $\mu$ we have
\begin{equation}\label{eq:path-in-nhood}
\{\bf{x}_{i,j}:\ 0\leq j \leq \ell_i\} \subseteq \O(\calO,\s_{n_i})
\end{equation}
for all sufficiently large $i$.
\end{quote}
\end{dfn}

This definition is made more complicated by the allowance of an arbitrary subsequence $n_1 < n_2 < \dots$.  This is because of cases in which there are some other subsequence $n_1' < n_2' < \dots$ and a w$^\ast$-neighbourhood $\calO$ such that $\O(\calO,\s_{n'_i}) = \emptyset$ for all $i$.  Such cases should still be called `connected' if any two sufficiently good models, for a sufficiently large value of $n$, can be joined by a $\delta$-path consisting of fairly good models.  This requirement simply ignores any values $n'$ for which $\s_{n'}$ admits no good models at all.  However, if there is a subsequence $n_1' < n'_2 < \dots$ as above, then there can be no sequences $\bf{x}_n,\bf{y}_n \in \X^{V_n}$ defined for \emph{all} integers $n$ which satisfy $P^{\s_n}_{\bf{x}_n},P^{\s_n}_{\bf{y}_n}\stackrel{\rm{weak}^\ast}{\to} \mu$, and we must pass to a subsequence which eventually avoids the $n'_i$s.  This will be made clearer by the proof of Proposition~\ref{prop:reform} below.

Definition~\ref{dfn:connected-model} can be re-written more directly in terms of connectedness properties of the model spaces $\O(\calO,\s_n)$.  At first sight, it seems similar to requiring that these spaces are $\delta$-connected for all sufficiently large $n$, but that impression is not quite correct.  In order to fix it, we need a slightly more complicated notion.

If $(Y,d_Y)$ is any metric space and $A\subseteq B \subseteq Y$ is a nested pair of subsets, then the pair $(A,B)$ is \textbf{relatively $\delta$-connected} (\textbf{according to $d_Y$}) if for any $x,y \in A$ there is a $\delta$-path from $x$ to $y$ contained in $B$.

\begin{prop}\label{prop:reform}
If $(\X^G,\mu,S,d)$ is a metric $G$-process, then the following are equivalent.
\begin{enumerate}
 \item It has connected model spaces rel $\S$.
\item For every $\delta > 0$ and every w$^\ast$-neighbourhood $\calO$ of $\mu$, there is a w$^\ast$-neighbourhood $\calO' \subseteq \calO$ such that the pair
\[\big(\O(\calO',\s_n),\O(\calO,\s_n)\big)\]
is relatively $\delta$-connected according to $d^{(V_n)}$ for all sufficiently large $n$.
\end{enumerate}
These properties are both implied by the following.
\begin{enumerate}
 \item[3.] For every $\delta > 0$, $\mu$ has a base of w$^\ast$-neighbourhoods $\calN$ with the property that the set $\O(\calN,\s_n)$ is $\delta$-connected for all sufficiently large $n$.
\end{enumerate}
\end{prop}

\begin{proof}
(1. $\Longrightarrow$ 2.) \quad If property 2 does not hold, then there are some $\delta$ and $\calO$ for which it fails.  Let $\calO_1 \supseteq \calO_2\supseteq \cdots$ be a base for the w$^\ast$-topology at $\mu$ such that $\calO_i \subseteq \calO$ for every $i$.  By the failure of property 2, there are integers $n_1 < n_2 < \dots$ and pairs $\{\bf{x}_i,\bf{y}_i\} \subseteq \O(\calO_i,\s_{n_i})$ for every $i$ which cannot be connected by $\delta$-paths that stay inside $\O(\calO,\s_{n_i})$.  This prevents the existence of $\delta$-paths from $\bf{x}_i$ to $\bf{y}_i$ satisfying~(\ref{eq:path-in-nhood}) in Definition~\ref{dfn:connected-model}.

\vspace{7pt}

(2. $\Longrightarrow$ 1.) \quad Let $\bf{x}_i,\bf{y}_i \in \X^{V_{n_i}}$ be a pair of sequences as in Definition~\ref{dfn:connected-model}.  Let $\calO_1 \supseteq \calO_2 \supseteq \dots$ be a base for the w$^\ast$-topology at $\mu$.

For each $j$, property 2 gives a sub-neighbourhood $\calO'_j \subseteq \calO_j$ such that the pair
\begin{equation}\label{eq:conn-pair}
\big(\O(\calO_j',\s_n),\O(\calO_j,\s_n)\big)
\end{equation}
is relatively $2^{-j}$-connected (according to $d^{(V_n)}$) for all sufficiently large $n$.

Since $P^{\s_{n_i}}_{\bf{x}_i},P^{\s_{n_i}}_{\bf{y}_i} \stackrel{\rm{weak}^\ast}{\to} \mu$, there are $i_1 < i_2 < \dots$ such that $\bf{x}_i,\bf{y}_i \in \O(\calO'_j,\s_{n_i})$ for all $i \geq i_j$.  Now the relative $2^{-j}$-connectedness of the pair~(\ref{eq:conn-pair}) implies that for each $i \geq i_j$ there is a $2^{-j}$-path
\[\{\bf{x}_i = \bf{x}_{i,0},\ \bf{x}_{i,1},\ \dots,\ \bf{x}_{i,\ell_i} = \bf{y}_i\} \subseteq \O(\calO_j,\s_{n_i}).\]
These paths verify Definition~\ref{dfn:connected-model} with
\[\delta_i := 2^{-\max\{j:\ i \geq i_j\}}.\]

\vspace{7pt}

(3. $\Longrightarrow$ 2.) \quad For any $\calO$, property 3 gives a sub-neighbourhood $\calN \subseteq \calO$ for which $\O(\calN,\s_n)$ is $\delta$-connected for all sufficiently large $n$, and this implies that the pair
\[\big(\O(\calN,\s_n),\O(\calO,\s_n)\big)\]
is relatively $\delta$-connected.
\end{proof}

In all the examples of connected model spaces that I know, one actually has property 3 above, which is easier to verify.  But I do not see a proof that these are equivalent, and I also do not know whether property 3 is isomorphism-invariant.

\begin{rmk}
Definition~\ref{dfn:connected-model} has a natural modification as follows.  Let us say that $(\X^G,\mu,S,d)$ has \textbf{uniformly connected model spaces rel $\S$} if there is a function $\ell:(0,1)\to \bbN$ for which following holds:
\begin{quote}
If $n_1 < n_2 < \dots$, and $\bf{x}_i,\bf{y}_i \in \X^{V_{n_i}}$ are two sequences satisfying
\[P^{\s_{n_i}}_{\bf{x}_i},P^{\s_{n_i}}_{\bf{y}_i} \stackrel{\rm{weak}^\ast}{\to} \mu,\]
then for every $\delta > 0$ and all sufficiently large $i$ there are $\delta$-paths \[\{\bf{x}_i = \bf{x}_{i,0},\ \bf{x}_{i,1},\ \dots,\ \bf{x}_{i,\ell(\delta)} = \bf{y}_i\} \subseteq \X^{V_{n_i}}\]
such that for any w$^\ast$-neighbourhood $\calO$ of $\mu$ we have
\[\{\bf{x}_{i,j}:\ 0\leq j \leq \ell(\delta)\} \subseteq \O(\calO,\s_{n_i})\]
for all sufficiently large $i$.
\end{quote}

That is, one requires that the lengths of the $\delta$-paths depend only on $\delta$, not on $n$.  It is easily shown that this variant is formally stronger than Definition~\ref{dfn:connected-model}

In the case of our principal examples, Bernoulli shifts, the proof below actually shows that this stronger property holds.  I do not know of any examples that have connected but not uniformly connnected model spaces. \fin
\end{rmk}

\section{Model-surjective factor maps and Theorem A}

As promised in the Introduction, for some groups $G$ the Popa factor does not preserve the connectedness of model spaces.  We now introduce a special kind of factor map which does always preserve this property, and show that all isomorphisms are factor maps of this kind.  They are defined in terms of good models for the associated graphical joinings, but we will prove an equivalent characterization in terms of AL approximations to the factor map.  We finish the subsection by proving Theorem A.

Later we show that complemented factor maps of Bernoulli shifts are also of this kind, which leads to the proof of Corollary D$^{\prime\prime}$.

Let
\[\Phi = \phi^G:(\X^G,\mu,S,d_\X)\to (\Y^G,\nu,S,d_\Y)\]
be a factor map of metric $G$-processes, and let
\[\l = \int_{\X^G} \delta_{(x,\Phi(x))}\,\mu(\d x)\]
be the associated graphical joining of $\mu$ and $\nu$.    Also, let $d$ be the Hamming average of the metrics $d_\X$ and $d_\Y$ on $\X\times \Y$.

\begin{dfn}[Model-surjective factor maps rel $\S$]\label{dfn:model-surj}
This factor map $\Phi$ is \textbf{model-surjective rel $\S$} if, whenever $n_1 < n_2 < \dots$ and $\bf{y}_i \in \Y^{V_{n_i}}$ is a sequence satisfying
\[P^{\s_{n_i}}_{\bf{y}_i} \stackrel{\rm{weak}^\ast}{\to} \nu,\]
there is another sequence $\bf{x}_i \in \X^{V_{n_i}}$ such that
\[P^{\s_{n_i}}_{(\bf{x}_i,\bf{y}_i)} \stackrel{\rm{weak}^\ast}{\to} \l,\]
where this last convergence refers to the w$^\ast$-topology arising from the product topology on $\X^G\times \Y^G$.
\end{dfn}

Clearly the sequence $\bf{x}_i$ obtained above must also satisfy $P^{\s_n}_{\bf{x}_i} \stackrel{\rm{weak}^\ast}{\to} \mu$. The idea behind Definition~\ref{dfn:model-surj} is that, if the joint empirical distribution converges to $\l$, then this forces $\bf{x}_i$ to `resemble' a pre-image of $\bf{y}_i$ under some model-space approximation to the factor map $\Phi$.

The need to consider a sofic sub-approximation $(\s_{n_i})_{i\geq 1}$ of $\S$ is similar to the case of Definition~\ref{dfn:connected-model}.

Beware that at this point, Definition~\ref{dfn:model-surj} requires a particular choice of the metrics $d_\X$ and $d_\Y$.  Once we have shown that isomorphisms are model-surjective (Proposition~\ref{prop:isos-surj} below), it will follow that Definition~\ref{dfn:model-surj} actually depends only on the measure theoretic structure of $\Phi$ as a factor map from $(\X^G,\mu,S)$ to $(\Y^G,\nu,S)$, and moreover that it can be extended unambiguously to factor maps of general $G$-systems (Corollary~\ref{cor:model-surj-iso-invar}).

First we need a simple lemma and corollary which relate the metrics $d^{(V_n)}$ and certain empirical distributions.

\begin{lem}\label{lem:dist-and-emp}
Let $\S$ be as before, and let $(\X,d)$ be a compact metric space.  For any two sequences $\bf{x}_n,\bf{z}_n \in \X^{V_n}$, the following are equivalent:
\begin{enumerate}
\item we have $d^{(V_n)}(\bf{x}_n,\bf{z}_n) \to 0$;
\item in the w$^\ast$ topology, any subsequential limit of the sequence
\[P^{\s_n}_{(\bf{x}_n,\bf{z}_n)} \in \Pr(\X^G\times \X^G)\]
is supported on the diagonal
\[\{(x,x):\ x \in \X^G\}.\]
\end{enumerate}
In case $P^{\s_n}_{\bf{x}_n}$ w$^\ast$-converges to some $\mu \in \Pr(\X^G)$, either of the above conditions implies that $P^{\s_n}_{\bf{z}_n}$ w$^\ast$-converges to the same limit.
\end{lem}

\begin{proof}
On $\X^G\times \X^G$, consider the function
\[F(x,x') := d(x_e,x'_e) \quad \hbox{for}\ x = (x_g)_g,\ x' = (x'_g)_g \in \X^G.\]
It is continuous, and a simple calculation gives
\[\int F\,\d P^{\s_n}_{(\bf{x}_n,\bf{z}_n)} = d^{(V_n)}(\bf{x}_n,\bf{y}_n).\]

Assuming condition (1), it follows that any subsequential limit
\[\l = \lim_{j\to\infty} P^{\s_{n_j}}_{(\bf{x}_{n_j},\bf{z}_{n_j})}\]
must satisfy $\int F\,\d\l = 0$, hence be supported on the set $\{(x,x'):\ x_e = x'_e\}$.  Since $\l$ is $S$-invariant (see~\cite[Lemma 3.2]{Aus--soficentadd}), it must be supported on the diagonal.

On the other hand, if (1) fails, then there are some $\delta > 0$ and some subsequence $n_1 < n_2 < \dots$ for which
\[d^{(V_{n_j})}(\bf{x}_{n_j},\bf{z}_{n_j})\to \delta \quad \hbox{as}\ j\to\infty.\]
By the sequential compactness of the w$^\ast$-topology on $\Pr(\X^G\times \X^G)$, there is a further subsequence for which the empirical distributions converge to some $\l$.  This $\l$ must then satisfy $\int F\,\d\l = \delta > 0$, and so cannot be supported on the diagonal.

Finally, any measure supported on the diagonal must have equal first and second marginals, so condition (2) clearly implies the last part of the lemma.
\end{proof}

\begin{cor}\label{cor:dist-and-emp}
Let
\[\Phi = \phi^G:(\X^G,\mu,S,d_\X) \to (\Y^G,\nu,S,d_\Y)\]
be a factor map of metric $G$-processes, and let
\[\l = \int_{\X^G}\delta_{(x,\Phi(x))}\,\mu(\d x)\]
be the resulting graphical joining of these two systems.  Suppose that $\bf{x}_n \in \X^{V_n}$ and $\bf{y}_n$, $\bf{w}_n \in \Y^{V_n}$ are sequence such that
\[P^{\s_n}_{(\bf{x}_n,\bf{y}_n)},P^{\s_n}_{(\bf{x}_n,\bf{w}_n)} \stackrel{\rm{weak}^\ast}{\to} \l.\]
Then $d_\Y^{(V_n)}(\bf{y}_n,\bf{w}_n)\to 0$.
\end{cor}

\begin{proof}
For each $n$, let
\[\theta_n := P^{\s_n}_{(\bf{x}_n,\bf{y}_n,\bf{w}_n)} \in \Pr((\X\times \Y \times \Y)^G),\]
and let
\[\theta := \lim_{j\to\infty} \theta_{n_j}\]
be any w$^\ast$-subsequential limit of these measures.

The projections of $\theta_n$ onto the two copies of $(\X\times \Y)^G$ are $P^{\s_n}_{(\bf{x}_n,\bf{y}_n)}$ and $P^{\s_n}_{(\bf{x}_n,\bf{w}_n)}$, so we know that the limiting measure $\theta$ has both of these projections equal to $\l$.  Hence $\theta$ is supported on
\[\{(x,y,y'):\ x \in \X^G,\ y = \phi^G(x) = y'\}.\]

Therefore the limit of $P^{\s_n}_{(\bf{y}_{n_j},\bf{w}_{n_j})}$ is supported on the diagonal in $(\Y\times \Y)^G$, and Lemma~\ref{lem:dist-and-emp} completes the proof.
\end{proof}

Now we can begin the study of model-surjectivity.  Graphical joinings give the easiest way to define this property, but it is useful to have a more `functional' characterization.  This can be given in terms of AL approximations to the factor map.

\begin{lem}\label{lem:model-surj-equiv}
For $\Phi = \phi^G$ as above, the following are equivalent.
\begin{enumerate}
\item The factor map $\Phi$ is model-surjective rel $\S$.
\item Suppose that $\psi_k \aL \phi$ rel $(\mu,d_\X,d_\Y)$. If $n_1 < n_2 < \dots$ and $\bf{y}_i \in \Y^{V_{n_i}}$ is a sequence satisfying
\[P^{\s_{n_i}}_{\bf{y}_i} \stackrel{\rm{weak}^\ast}{\to} \nu,\]
then there is a sequence $\bf{x}_i \in \X^{V_{n_i}}$ such that we have
\[P^{\s_{n_i}}_{\bf{x}_i} \stackrel{\rm{weak}^\ast}{\to} \mu\]
and
\[d_\Y^{(V_{n_i})}\big(\bf{y}_i,\psi^{\s_{n_i}}_{k_i}(\bf{x}_i)\big) \to 0\]
whenever $k_1 \leq k_2 \leq \dots $ grows sufficiently slowly.
\end{enumerate}
\end{lem}

\begin{proof}
Let $\psi_k \aL \phi$ be an AL approximating sequence rel $(\mu,d_\X,d_\Y)$.  Also, let $\xi:\X^G\to \X$ be the projection onto the $\{e\}$-coordinate, so $\xi^G = \rm{id}_{\X^G}$.  An easy check (or see~\cite[Corollary 4.7]{Aus--soficentadd}) gives that
\begin{equation}\label{eq:pair-AL-conv}
(\xi,\psi_k) \aL (\xi,\phi) \quad \hbox{rel}\ (\mu,d_\X,d),
\end{equation}
where $d$ is the Hamming average of $d_\X$ and $d_\Y$ on $\X\times \Y$, and $(\xi,\phi)$ denotes the map
\[\X^G\to \X\times \Y:x\mapsto (\xi(x),\phi(x)).\]

Let $n_1 < n_2 < \dots$, and let $\bf{y}_i \in \Y^{V_{n_i}}$ be a sequence whose empirical distributions tend to $\nu$.  By re-labeling the sofic sub-approximation $(\s_{n_i})_{i\geq 1}$ if necessary, we may assume that $n_i = i$ for all $i$, and hence write $\bf{y}_n$ as an element of $\Y^{V_n}$.

Now suppose $\bf{x}_n \in \X^{V_n}$ is another sequence whose empirical distributions tend to $\mu$.  By~(\ref{eq:pair-AL-conv}) and Corollary~\ref{cor:approx-by-Lip-maps}, it follows that
\[P^{\s_n}_{(\bf{x}_n,\psi_{k_n}^{\s_n}(\bf{x}_n))} = P^{\s_n}_{(\xi,\psi_{k_n})^{\s_n}(\bf{x}_n)} \stackrel{\rm{weak}^\ast}{\to} \l,\]
provided the sequence $k_1 \leq k_2 \leq \dots$ grows sufficiently slowly.

The result follows because Lemma~\ref{lem:dist-and-emp} (respectively Corollary~\ref{cor:dist-and-emp}) shows that
\[P^{\s_n}_{(\bf{x}_n,\bf{y}_n)} \stackrel{\rm{weak}^\ast}{\to} \l\]
if (respectively only if)
\[d^{(V_n)}\big((\bf{x}_n,\bf{y}_n),(\bf{x}_n,\psi_{k_n}^{\s_n}(\bf{x}_n))\big)= \frac{1}{2} d_\Y^{(V_n)}\big(\bf{y}_n,\psi_{k_n}^{\s_n}(\bf{x}_n)\big) \to 0.\]
\end{proof}

\begin{prop}\label{prop:compose-still-surj}
If
\begin{multline*}
\Phi:(\X^G,\mu,S,d_\X)\to (\Y^G,\nu,S,d_\Y) \\ \hbox{and} \quad \t{\Phi}:(\Y^G,\nu,S,d_\Y) \to (\Z^G,\theta,S,d_\Z)
\end{multline*}
are both model-surjective rel $\S$, then so is their composition.
\end{prop}

\begin{proof}
As in the proof above, after passing to a sofic sub-approximation we may assume that $\bf{z}_n \in \Z^{V_n}$ is a sequence whose empirical distributions tend to $\theta$.  By the two assumed instances of model-surjectivity rel $\S$, we may find first a sequence $\bf{y}_n \in \Y^{V_n}$ and then a sequence $\bf{x}_n \in \X^{V_n}$ such that
\[P^{\s_n}_{(\bf{x}_n,\bf{y}_n)} \stackrel{\rm{weak}^\ast}{\to} \l \quad \hbox{and}\quad P^{\s_n}_{(\bf{y}_n,\bf{z}_n)} \stackrel{\rm{weak}^\ast}{\to} \t{\l},\]
where $\l$ and $\t{\l}$ are the graphical joinings associated to $\Phi$ and $\t{\Phi}$ respectively.

Now let
\[\l_1 = \lim_{j\to\infty} P^{\s_{n_j}}_{(\bf{x}_{n_j},\bf{y}_{n_j},\bf{z}_{n_j})}\]
be any subsequential limit of the triple empirical distributions.  Since its projection onto the first two coordinates must be $\l$ and its projection onto the second two coordinates must be $\t{\l}$, it is supported on the set
\[\big\{\big(x,\Phi(x),\t{\Phi}(\Phi(x))\big):\ x \in \X^G\big\},\]
and hence it must be the full graphical joining associated to $(\Phi,\t{\Phi}\circ \Phi)$.  Therefore any subsequence of
\[P^{\s_n}_{(\bf{x}_n,\bf{z}_n)}\]
converges to the graphical joining associated to $\t{\Phi}\circ \Phi$, as required.
\end{proof}

\begin{prop}\label{prop:isos-surj}
If $\Phi = \phi^G$ is an isomorphism, then it is model-surjective for any sofic approximation $\S$.
\end{prop}

\begin{proof}
Let $\t{\phi}$ be such that $\Phi^{-1} = \t{\phi}^G$, let $\xi:\Y^G\to \Y$ be the projection to the $\{e\}$-indexed coordinate, and let $\t{\psi}_m \aL \t{\phi}$ rel $(\nu,d_\Y,d_\X)$.  Let $\l$ be the graphical joining of $\Phi$. Then we also have
\begin{equation}\label{eq:combine}
(\t{\psi}_m,\xi) \aL (\t{\phi},\xi) \quad \hbox{rel}\ (\nu,d_\Y,d)
\end{equation}
(see again~\cite[Corollary 4.7]{Aus--soficentadd} for a careful proof of this).

As in the proofs above, to show the model-surjectivity of $\Phi$ we may pass to a sofic sub-approximation, and so suppose that $\bf{y}_n \in \Pr(\Y^{V_n})$ is a sequence whose empirical distributions converge to $\nu$.  Then~(\ref{eq:combine}) and Lemma~\ref{lem:approx-by-Lip-maps} give that
\[P^{\s_n}_{\left(\t{\psi}_{m_n}^{\s_n}(\bf{y}_n),\bf{y}_n\right)}\stackrel{\rm{weak}^\ast}{\to} \l\]
provided $(m_n)_{n\geq 1}$ grows sufficiently slowly.

Letting $\bf{x}_n := \t{\psi}_{m_n}^{\s_n}(\bf{y}_n)$, this completes the proof.
\end{proof}

The first important consequence of Propositions~\ref{prop:compose-still-surj} and~\ref{prop:isos-surj} is the following.

\begin{cor}\label{cor:model-surj-iso-invar}
If
\begin{center}
$\phantom{i}$\xymatrix{
(\X^G,\mu,S,d_\X) \ar_\Phi[d]\ar^-\cong[r] & (\X_1^G,\mu_1,S,d_{\X_1}) \ar^\Psi[d]\\
(\Y^G,\nu,S,d_\Y) \ar^-\cong[r] & (\Y_1^G,\nu_1,S,d_{\Y_1})
}
\end{center}
is a commutative square of factor maps in which the horizontal arrows are isomorphisms and $\Phi$ is model-surjective rel $\S$, then $\Psi$ is also model-surjective rel $\S$.

In particular, model-surjectivity rel $\S$ in Definition~\ref{dfn:model-surj} is independent of the choice of generating metrics $d_\X$ and $d_\Y$.\qed
\end{cor}

We now have all the ingredients needed to prove Theorem A.

\begin{proof}[Proof of Theorem A]
Let $n_1 < n_2 < \dots$, and let $\bf{y}_i,\bf{w}_i \in \Y^{V_{n_i}}$ be sequences such that
\[P^{\s_{n_i}}_{\bf{y}_i},P^{\s_{n_i}}_{\bf{w}_i} \stackrel{\rm{weak}^\ast}{\to} \nu.\]
We must show that these pairs may be connected by $o(1)$-paths consisting of good models, as in Definition~\ref{dfn:connected-model}.  By passing to the sofic sub-approximation $(\s_{n_i})_{i\geq 1}$, we may relabel all these sequences and so assume that $i=n_i$ for all $i$, and hence write the index as $n$ itself.

Now let $\bf{x}_n,\bf{z}_n \in \X^{V_n}$ be the sequences given by the model-surjectivity of $\Phi$ applied to the sequences $\bf{y}_n$ and $\bf{w}_n$, respectively. Since $(\X^G,\mu,S,d_\X)$ has connected model space rel $\S$, there are parameters $\delta_n \downarrow 0$ and a sequence of $\delta_n$-paths
\[\bf{x}_n = \bf{x}_{n,0},\bf{x}_{n,1},\dots,\bf{x}_{n,\ell_n} = \bf{z}_n\]
which are eventually contained in $\O(\calO,\s_n)$ for any w$^\ast$-neighbourhood $\calO$ of $\mu$.

Also, let $\psi_k \aL \phi$ rel $(\mu,d_\X,d_\Y)$.  Corollary~\ref{cor:combined-good-behav} gives parameters $\eps_k \downarrow 0$ and $K_k < \infty$, a sequence of w$^\ast$-neighbourhoods $\calO_1 \supseteq \calO_2 \supseteq \dots$ of $\mu$, and a base of w$^\ast$-neighbourhoods $\calN_1 \supseteq \calN_2 \supseteq \dots$ at $\nu$, such that for every $k$ we have both
\begin{equation}\label{eq:alm-Lip-ass}
\psi_k^{\s_n}|\O(\calO_k,\s_n) \quad \hbox{is}\ \eps_k\hbox{-almost}\ K_k\hbox{-Lipschitz for all sufficiently large $n$}
\end{equation}
and
\begin{equation}\label{eq:approx-by-Lip-maps}
\psi_k^{\s_n}\big(\O(\calO_k,\s_n)\big) \subseteq \O(\calN_k,\s_n) \quad \hbox{for all sufficiently large $n$}.
\end{equation}

Now choose a sequence $k_1 \leq k_2 \leq \dots$ growing so slowly that all of the following hold:
\begin{enumerate}
\item[i)] we have $\eps_{k_n} + K_{k_n}\delta_n \to 0$ as $n\to\infty$ (this is possible because $\eps_k \to 0$ and $\delta_n \to 0$);
\item[ii)] we have
\[\{\bf{x}_{n,0},\bf{x}_{n,1},\dots,\bf{x}_{n,\ell_n}\} \subseteq \O(\calO_{k_n},\s_n)\]
for all sufficiently large $n$;
\item[iii)] we have
\[d_\Y^{(V_n)}\big(\bf{y}_n,\psi_{k_n}^{\s_n}(\bf{x}_n)\big),\ d_\Y^{(V_n)}\big(\bf{w}_n,\psi_{k_n}^{\s_n}(\bf{z}_n)\big) \to 0 \quad \hbox{as}\ n\to\infty,\]
as promised by condition 2 in Lemma~\ref{lem:model-surj-equiv}.
\item[iv)] we have
\[\psi_{k_n}^{\s_n}|\O(\calO_{k_n},\s_n) \quad \hbox{is}\ \eps_{k_n}\hbox{-almost}\ K_{k_n}\hbox{-Lipschitz}\]
for all sufficiently large $n$, as is possible by~(\ref{eq:alm-Lip-ass})
\item[v)] we have
\[\psi^{\s_n}_{k_n}\big(\O(\calO_{k_n},\s_n)\big) \subseteq \O(\calN_{k_n},\s_n)\]
for all sufficiently large $n$, as is possible by~(\ref{eq:approx-by-Lip-maps}).
\end{enumerate}

Finally, define
\[\bf{y}_{n,i} := \psi_{k_n}^{\s_n}(\bf{x}_{n,i})\]
for each $n$ and each $i = 0,1,\dots,\ell_n$.

By properties (ii) and (iv), the sequence
\[\bf{y}_{n,0},\bf{y}_{n,1},\dots,\bf{y}_{n,\ell_n}\]
is a $(\eps_{k_n} + K_{k_n}\delta_n)$-path according to $d_\Y^{(V_n)}$ for all sufficiently large $n$.  Combined with properties (i) and (iii), this shows that
\[\bf{y}_n,\bf{y}_{n,0},\bf{y}_{n,1},\dots,\bf{y}_{n,\ell_n},\bf{w}_n\]
is a $\delta_n'$-path from $\bf{y}_n$ to $\bf{w}_n$ for some sequence of parameters $\delta'_n \downarrow 0$.

Finally, properties (ii) and (v) imply that these image paths are eventually contained in $\O(\calN,\s_n)$ for every w$^\ast$-neighbourhood $\calN$ of $\nu$, because $\calN_1 \supseteq \calN_2 \supseteq \dots$ is a base at $\nu$.
\end{proof}

\section{Connected model spaces for inverse limits}

This section gives the proof of Theorem B.

\begin{proof}[Proof of Theorem B]
Up to isomorphism, the setting of this theorem may be represented as follows.  Let $(\X_i^G,\mu_i,S)$, $i=1,2,\dots$ be an infinite sequence of $G$-processes, and let
\[\l \in \Pr(\X_1^G\times \X_2^G\times \cdots)\]
be a joining of all of them.  Let $\pi_k:\prod_{i\geq 1}\X_i \to \prod_{i=1}^k \X_i$ be the coordinate projection for each $k$, and let $\l_k := \pi^G_{k\ast}\l$.  We assume that the $G$-system
\[\Big(\prod_{i=1}^k \X_i^G,\l_k,S\Big)\]
has connected model spaces rel $\S$ for each $k$, and must prove the same for the infinite joining $\l$.  This is easiest using the reformulation in condition 2 of Proposition~\ref{prop:reform}.

For each $i$, let $d_{\X_i}$ be a compact generating metric of diameter at most $1$ for the space $\X_i$. Define metrics $d_k$ on $\prod_{i=1}^k \X_i$ and $d$ on $\prod_{i\geq 1}\X_i$ by
\[d_k\big((x_i)_{i=1}^k,(x'_i)_{i=1}^k\big) := \sum_{i=1}^k 2^{-i}d_{\X_i}(x_i,x_i')\]
and
\[d\big((x_i)_{i\geq 1}^k,(x'_i)_{i\geq 1}\big) := \sum_{i\geq 1} 2^{-i}d_{\X_i}(x_i,x_i'),\]
so these generate the compact product topologies on their respective spaces.

Now suppose that $\calO$ is a w$^\ast$-neighbourhood of $\l$.  By shrinking it if necessary, we may assume that it has the form
\[\big\{\theta:\ \pi^G_{k\ast}\theta \in \calO_1\big\}\]
for some $k \in \bbN$ and some w$^\ast$-neighbourhood $\calO_1$ of $\l_k$, since sets of this form are a base of w$^\ast$-neighbourhoods around $\l$.  Under this assumption, we obtain also
\begin{multline}\label{eq:omega-is-projected}
\O(\calO,\s_n) = \Big\{\bf{x} \in \Big(\prod_{i\geq 1}\X_i\Big)^{V_n}:\ \pi_k^{\s_n}(\bf{x}) \in \O(\calO_1,\s_n)\Big\} \\
= \O(\calO_1,\s_n) \times \Big(\prod_{i\geq n+1} \X_i\Big)^{V_n} \quad \hbox{for each}\ n.
\end{multline}

Since the process defined by $\l_k$ has connected model spaces rel $\S$, Proposition~\ref{prop:reform} gives another w$^\ast$-neighbourhood $\calO_1' \subseteq \calO_1$ of $\l_k$ such that the pair
\[\big(\O(\calO_1',\s_n),\O(\calO_1,\s_n)\big)\]
is relatively $\delta$-connected according to $d_k^{(V_n)}$ for all sufficiently large $n$.

Let
\[\calO' := \big\{\theta:\ \pi^G_{k\ast}\theta \in \calO_1'\big\} \subseteq \calO.\]
Equation~(\ref{eq:omega-is-projected}) has an obvious analog for $\calO'$ and $\calO_1'$.  Combining these, it follows that the pair
\[\big(\O(\calO',\s_n),\O(\calO,\s_n)\big)\]
is also relatively $\delta$-connected according to $d^{(V_n)}$ for all sufficiently large $n$.  This verifies condition 2 in Proposition~\ref{prop:reform} for the process defined by $\l$.
\end{proof}

\section{Connected model spaces for Bernoulli systems}

This subsection proves Theorem C.  We actually prove the slightly stronger property (3) from Proposition~\ref{prop:reform}. Let $(\X,d)$ be a compact metric space of diameter at most $1$, and let $\nu \in \Pr(\X)$.

This section makes several simple appeals to the phenomenon of measure concentration for product measures and Hamming metrics: see, for instance,~\cite{Led--book} for a dedicated exposition and~\cite[Chapter 3$\frac{1}{2}$]{Gro01} for a geometrically-flavoured overview.  The specific result that we need is the following: see~\cite[Corollary 1.17]{Led--book}.

\begin{prop}\label{prop:Tal}
For any $\delta > 0$ there is a $\b > 0$ such that
\[\sup_f\nu^{\times n}\Big\{\Big|f - \int f\,\d\nu^{\times n}\Big| \geq \delta\Big\} = O(\rme^{-\b n})\]
for all $n\geq 1$, where the supremum is over all functions $f:\X^n \to \bbR$ which are $1$-Lipschitz for the normalized Hamming distance $d^{(n)}$. \qed
\end{prop}

Next we identify certain special neighbourhoods of the product measure $\nu^{\times G}$.  Let $F$ be a finite subset of $G$, and for each $D \subseteq F$ define an operator $E_D$ on $C(\X^F)$ as follows.  First, if $F = \{g_1,\ldots,g_m\}$ and $D = F\setminus \{g_j\}$ for some $j \leq m$, then
\[E_Df(x) := \int_X f(x_{g_1},\ldots,x_{g_{j-1}},y,x_{g_{j+1}},\ldots,x_{g_m})\ \nu(\d y).\]
In general, if $D = F \setminus \{g_{j_1},\ldots,g_{j_k}\}$, then
\[E_D := E_{F\setminus \{g_{j_1}\}}\circ \cdots \circ E_{F\setminus \{g_{j_k}\}},\]
where the order of this composition is unimportant.  This $E_D$ is the operator of conditional expectation with respect to $\nu^{\times F}$ onto the $\s$-algebra of Borel subsets of $\X^F$ that depend only on coordinates in $D$. It follows that
\begin{equation}\label{eq:ED-meas-pres}
\int f\,\d \nu^{\times F} = \int E_Df\,\d \nu^{\times F} \quad \forall f \in C(\X^F).
\end{equation}

Now suppose that $\F \subseteq C(\X^G)$ is a family of $F$-local functions.  Let us call it \textbf{hereditary} if every member of $\F$ is $1$-Lipschitz according to $d_\X^{(F)}$ and if
\[f \in \F \quad \Longrightarrow \quad E_Df \in \F \quad \forall D\subseteq F.\]
It is easily checked that each $E_D$ preserves the property of being $1$-Lipschitz with respect to $d_\X^{(F)}$. A \textbf{hereditary neighbourhood} of $\nu^{\times G}$ is a w$^\ast$-neighbourhood of the form
\begin{equation}\label{eq:hered-nhood}
\calO := \Big\{\theta\in \Pr(\X^G):\ \int f\,\d\theta \approx_\eps \int f\,\d\nu^{\times G}\ \forall f\in\F\Big\}.
\end{equation}
for some finite $F\subseteq G$, some $\eps > 0$, and some finite hereditary family $\F$ of $F$-local continuous functions.

\begin{prop}\label{prop:B+}
If $\calO$ is a hereditary neighbourhood of $\nu^{\times G}$ and $\delta > 0$, then the set $\O(\calO,\s_n)$ is $\delta$-connected according to $d^{(V_n)}$ for all sufficiently large $n$.
\end{prop}

We will see the importance of assuming that $\calO$ is hereditary during the course of the proof.

\begin{proof}[Proof of Theorem C from Proposition~\ref{prop:B+}]
Hereditary neighbourhoods form a base for the w$^\ast$-topology at $\mu$.  Therefore Proposition~\ref{prop:B+} verifies condition 3 in Proposition~\ref{prop:reform}, which implies connected model spaces.
\end{proof}

Proposition~\ref{prop:B+} will be proved by showing how any pair of points in $\O(\calO,\s_n)$ may be connected by a `random' $\delta$-path, provided $n$ is sufficiently large.

We insert randomness into the proof as follows.  Fix $\k \in (0,1)$, to be specified later. For each $n$, let $(\bs{\xi}_{n,t})_{t\geq 0}$ be a discrete-time random walk on $\X^{V_n}$ with the following transition probabilities:
\begin{eqnarray}\label{eq:law-of-X_k}
\sfP(\bs{\xi}_{n,t+1} \in \,\cdot\,|\,\bs{\xi}_{n,t}) = \bigtimes_{v\in V_n}\big(\k\delta_{\xi_{n,t,v}} + (1 - \k)\nu\big),
\end{eqnarray}
where we write $\bs{\xi}_{n,t} = (\xi_{n,t,v})_{v \in V_n}$.  Thus, $\bs{\xi}_{n,t+1}$ is obtained from $\bs{\xi}_{n,t}$ by considering each coordinate in $\X^{V_n}$ independently, and either re-sampling it from the distribution $\nu$ with probability $1-\k$, or leaving it unchanged with probability $\k$.  Let $\sfP_n^\bf{x}$ be the law of $(\bs{\xi}_{n,t})_{t \geq 0}$ on $\X^{V_n}\times \X^{V_n} \times \cdots$ conditioned on starting from $\bs{\xi}_{n,0} = \bf{x}$, and let $\sfE_n^\bf{x}$ denote expectation with respect to $\sfP_n^\bf{x}$.

We will find a $\delta$-path between two good models $\bf{x}$ and $\bf{y}$ by starting a copy of this random walk at each of $\bf{x}$ and $\bf{y}$, and showing that after a certain bounded time the following hold:
\begin{itemize}
\item[(i)] these random walks have probably stayed inside the set of good models,
\item[(ii)] they have probably taken only steps smaller than $\delta$ in the Hamming distance, and
\item[(iii)] they can be coupled in such a way that with high probability they end up close to each other.
\end{itemize}

We break the necessary estimates into three separate lemmas.

\begin{lem}\label{lem:staying-in-model-spaces}
For any hereditary neighbourhood $\calO$ and any $t \in \bbN$, we have
\[\inf_{\bf{x} \in \O(\calO,\s_n)}\sfP^\bf{x}_n\big\{\bs{\xi}_{n,t} \in \O(\calO,\s_n)\big\}\to 1 \quad \hbox{as}\ n\to\infty.\]
\end{lem}

\begin{proof}
Let $\calO$ be as in~(\ref{eq:hered-nhood}) for some $\eps > 0$ and some hereditary family $\F$ of $F$-local functions.

Since $\F$ is finite, it suffices to show that for any one $f \in \F$ we have
\[\inf_{\bf{x} \in \O(\calO,\s_n)}\sfP^\bf{x}_n\Big\{\int f\,\d P^{\s_n}_{\bs{\xi}_{n,t}} \approx_\eps \int f\,\d \nu^{\times G} \Big\} \to 1.\]

We break this estimate into two further steps.

\vspace{7pt}

\emph{Step 1.}\quad Observe that
\[\sfE_n^\bf{x}\int f\,\d P^{\s_n}_{\bs{\xi}_{n,t}} = \frac{1}{|V_n|}\sum_{v \in V_n}\sfE_n^\bf{x}f\big(\Pi^{\s_n}_v(\bs{\xi}_{n,t})\big) = \frac{1}{|V_n|}\sum_{v \in V_n}\sfE_n^\bf{x}f\big((\bs{\xi}_{n,t,\s_n^g(v)})_{g \in F}\big).\]
Since $F$ is finite and $\S$ is a sofic approximation, as $n\to\infty$ it holds with high probability in the choice of $v \in V_n$ that
\begin{equation}\label{eq:distinct-pts}
\hbox{the points} \  \s_n^g(v) \  \hbox{for} \  g \in F \  \hbox{are distinct}.
\end{equation}
For such $v$, iterating the formula~(\ref{eq:law-of-X_k}) gives
\[\sfE_n^\bf{x}f\big((\bs{\xi}_{n,t,\s_n^g(v)})_{g \in F}\big) = \sum_{D\subseteq F} \k^{t|D|}(1-\k^t)^{|F| - |D|}E_Df(\Pi^{\s_n}_v(\bf{x})).\]

Let $D$ be a random subset of $F$ which contains each element of $F$ independently with probability $\k^t$.  Let $\sfP'$ be its law and let $\sfE'$ be expectation with respect to $\sfP'$.  Then the above leads to
\[\sfE_n^\bf{x}\int f\,\d P^{\s_n}_{\bs{\xi}_{n,t}} = \sfE'\Big(\frac{1}{|V_n|}\sum_{v \in V}E_Df(\Pi^{\s_n}_v(\bf{x}))\Big) + o(1),\]
where the $o(1)$-correction results from those few vertices $v$ which fail the requirement~(\ref{eq:distinct-pts}).

Therefore
\begin{eqnarray*}
&&\Big|\sfE_n^\bf{x}\int f\,\d P^{\s_n}_{\bs{\xi}_{n,t}} - \int f\,\d \nu^{\times G}\Big|\\
&&\leq \sfE'\Big|\frac{1}{|V_n|}\sum_{v \in V}E_Df(\Pi^{\s_n}_v(\bf{x})) - \int f\,\d \nu^{\times G}\Big| + o(1)\\
&&= \sfE'\Big|\frac{1}{|V_n|}\sum_{v \in V}E_Df(\Pi^{\s_n}_v(\bf{x})) - \int E_Df\,\d \nu^{\times G}\Big| + o(1)\\
&&= \sfE'\Big|\int E_Df\,\d (P^{\s_n}_\bf{x} - \nu^{\times G})\Big| + o(1)\\
&&= \sfE'\Big(\Big|\int E_Df\,\d (P^{\s_n}_\bf{x} - \nu^{\times G})\Big|\cdot 1_{\{D \neq \emptyset\}}\Big) + o(1),
\end{eqnarray*}
where the equality of second and third lines uses~(\ref{eq:ED-meas-pres}), and the last equality uses that $E_\emptyset f$ is constant and equal to $\int f\,\d \nu^{\times F}$.  Since $\bf{x} \in \O(\calO,\s_n)$ with $\calO$ as in~(\ref{eq:hered-nhood}), the last line above is strictly bounded by
\[\sfP'\{D \neq \emptyset\}\cdot \eps + o(1) = (1 - (1-\k^t)^{|F|})\eps + o(1),\]
where the $o(1)$-correction depends only on the sofic approximation $\S$ and on $\|f\|_\infty$.  Therefore this bound is strictly less than, say, $(1 - \frac{1}{2}(1-\k^t)^{|F|})\eps$ for all sufficiently large $n$.

\vspace{7pt}

\emph{Step 2.}\quad On the other hand, we may regard the quantity
$\int f\,\d P^{\s_n}_{\bs{\xi}_{n,t}}$ as a random variable on the probability space
\[\big(\X^{V_n},\sfP_n^\bf{x}\{\bs{\xi}_{n,t}\in\,\cdot\,\}\big).\]
Using the initial condition $\bs{\xi}_{n,0} = \bf{x}$ and the transition probabilities~(\ref{eq:law-of-X_k}), the probability measure here is equal to
\[\bigtimes_{v\in V_n}\big(\k^t\delta_{x_v} + (1 - \k^t)\nu\big),\]
which is a product measure on $\X^{V_n}$.  By the definition of the empirical distribution, and recalling that $f$ is $F$-local and $1$-Lipschitz according to $d_\X^{(F)}$, this random variable is an $|F|$-Lipschitz function on this product space.  Therefore Proposition~\ref{prop:Tal} gives some $\b > 0$, depending only on the ratio $(1-\k^t)^{|F|}\eps/|F|$, such that
\[\sfP_n^\bf{x}\Big\{\Big|\int f\,\d P^{\s_n}_{\bs{\xi}_{n,t}} - \sfE_n^\bf{x}\int f\,\d P^{\s_n}_{\bs{\xi}_{n,t}}\Big| \geq \frac{1}{2}(1-\k^t)^{|F|})\eps\Big\} = O\big(\rme^{-\b|V_n|}\big).\]
Thus this probability tends to $0$ as $n\to\infty$ uniformly in $\bf{x}$.

\vspace{7pt}

Combining the estimates from Steps 1 and 2 completes the proof.
\end{proof}

\begin{lem}\label{lem:small-steps}
Suppose that $1 - \k < \delta$.  For any $t \in \bbN \cup \{0\}$, the random walks constructed above satisfy
\[\inf_{\bf{x} \in \X^{V_n}}\sfP^\bf{x}_n\big\{d^{(V_n)}(\bs{\xi}_{n,t},\bs{\xi}_{n,t+1}) < \delta\big\}\to 1 \quad \hbox{as}\ n\to \infty.\]
\end{lem}

\begin{proof}
By time-homogeneity, it suffices to prove this when $t= 0$.

Since $\rm{diam}(\X,d) \leq 1$, we always have
\[d^{(V_n)}(\bf{x},\bs{\xi}_{n,1}) \leq \frac{1}{|V_n|}\sum_{v \in V_n}1_{\{x_v \neq \xi_{n,1,v}\}}.\]
This is an average of indicator functions of independent events, all of them having probability at most $1 - \k < \delta$.  Therefore another appeal to Proposition~\ref{prop:Tal} (or just the special case of a Chernoff bound) gives a $\b > 0$ for which
\[\sfP_n^\bf{x}\big\{d^{(V_n)}(\bf{x},\bs{\xi}_{n,1}) \geq \delta\big\} = O(\rme^{-\b|V_n|}).\]
Since the right-hand bound is independent of $\bf{x}$, this completes the proof.
\end{proof}

\begin{lem}\label{lem:coupling}
If $s \in \bbN$ is so large that $\k^s < \delta/4$, then the following holds.  For any $n \in \bbN$ and any $\bf{x},\bf{y} \in \X^{V_n}$, the distributions
\[\sfP_n^{\bf{x}}\big\{(\bs{\xi}_{n,0},\dots,\bs{\xi}_{n,s}) \in \cdot \big\}\quad \hbox{and} \quad \sfP_n^{\bf{y}}\big\{(\bs{\zeta}_{n,0},\dots,\bs{\zeta}_{n,s}) \in \cdot \big\}\]
of the random walks started at $\bf{x}$ and $\bf{y}$ up to the finite time-horizon $s$ have a coupling $\bbQ$ such that
\begin{equation}\label{eq:coupled-close}
\bbQ\big\{d^{(V_n)}(\bs{\xi}_{n,s},\bs{\zeta}_{n,s}) < \delta\big\} > 1/2.
\end{equation}
\end{lem}

\begin{proof}
By an induction on $s$ using~(\ref{eq:law-of-X_k}), we have
\[\sfP^\bf{x}_n\{\bs{\xi}_{n,s}\in\cdot\} = \bigtimes_{v\in V_n}\big(\k^s\delta_{x_v} + (1 - \k^s)\nu\big),\]
and similarly for $\sfP^\bf{y}_n\{\bs{\zeta}_{n,s} \in \cdot\}$.

Consider a random triple $(\bs{\a},\bs{\xi},\bs{\zeta})$ of elements of $\X^{V_n}$ with law constructed as follows.  First, choose $\bs{\a}$ from the law $\nu^{\times V_n}$.  Then, for each $v \in V_n$ independently, choose two random bits $\eta_v,\omega_v \in \{0,1\}$ independently, each equal to $1$ with probability $\k^s$.  Finally, for each $v \in V_n$, set
\[\bs{\xi}_v := \left\{\begin{array}{ll}\bs{\a}_v &\quad \hbox{if}\ \eta_v = 0\\ x_v& \quad \hbox{if}\ \eta_v = 1\end{array}\right. \quad \hbox{and} \quad \bs{\zeta}_v := \left\{\begin{array}{ll}\bs{\a}_v &\quad \hbox{if}\ \o_v = 0\\ y_v& \quad \hbox{if}\ \o_v = 1.\end{array}\right.\]
Letting $\l$ be the joint distribution of $(\bs{\a},\bs{\xi},\bs{\zeta})$, it follows that
\[\l\{\bs{\xi} \in \cdot\} = \sfP^\bf{x}_n\{\bs{\xi}_{n,s}\in\cdot\}, \quad \l\{\bs{\zeta} \in \cdot\} = \sfP^\bf{y}_n\{\bs{\zeta}_{n,s} \in \cdot\},\]
and
\[\int d^{(V_n)}(\bs{\zeta},\bs{\xi})\,\d\l \leq 2\k^s < \delta/2.\]

Thus, under $\l$, the pair of random variables $(\bs{\xi},\bs{\zeta})$ are a coupling of $(\bs{\xi}_{n,s},\bs{\zeta}_{n,s})$ under which the probability of the event $\{d^{(V_n)}(\bs{\xi}_{n,x},\bs{\zeta}_{n,s}) < \delta\}$ is greater that $1/2$, by Chebyshev's Inequality.  Now we can choose any extension of this to a coupling $\sfQ$ of the whole random trajectories $(\bs{\xi}_{n,0},\dots,\bs{\xi}_{n,s})$ and $(\bs{\zeta}_{n,0},\dots,\bs{\zeta}_{n,s})$: for instance, we can couple them relatively independently over the given coupling of the end-states $\bs{\xi}_{n,s}$ and $\bs{\zeta}_{n,s}$.
\end{proof}

\begin{rmk}
Using Chernoff's Inequality for the random sum $\sum_v (\eta_v + \o_v)$, one can actually improve~(\ref{eq:coupled-close}) to a lower bound of the form $1 - O(\rme^{-\b|V_n|})$, but we will not need this. \fin
\end{rmk}

\begin{proof}[Proof of Proposition~\ref{prop:B+}]
Let $\calO$ be a hereditary w$^\ast$-neighbourhood of $\nu^{\times G}$ and let $\delta \in (0,1)$.  Choose some $\k \in (1 - \delta,1)$, and then choose $s \in \bbN$ so that $\k^s < \delta/4$.

Having chosen $s$, let $n$ be so large that
\begin{equation}\label{eq:stays-good}
\inf_{\bf{x} \in \O(\calO,\s_n)}\sfP^\bf{x}_n\big\{\bs{\xi}_{n,t} \in \O(\calO,\s_n) \ \forall t=1,2,\dots,s\big\} > 3/4
\end{equation}
and
\begin{equation}\label{eq:small-steps}
\inf_{\bf{x} \in \X^{V_n}}\sfP^\bf{x}_n\big\{d^{(V_n)}(\bs{\xi}_{n,t},\bs{\xi}_{n,t+1}) < \delta\ \forall t=0,1,\dots,s-1\big\} > 3/4.
\end{equation}
This is possible by Lemmas~\ref{lem:staying-in-model-spaces} and~\ref{lem:small-steps}, respectively.  We will show that $\O(\calO,\s_n)$ is $\delta$-connected for any such $n$.

Suppose that $\bf{x},\bf{y} \in \O(\calO,\s_n)$, and let $\sfQ$ be a coupling of two trajectories of the random walk up to time $s$, one starting from $\bf{x}$ and the other from $\bf{y}$, as constructed in Lemma~\ref{lem:coupling}.  Then the conjunction of the bounds~(\ref{eq:stays-good}),~(\ref{eq:small-steps}) and~(\ref{eq:coupled-close}) shows that the event
\begin{multline*}
\Big\{\bf{x} = \bs{\xi}_{n,0},\bs{\xi}_{n,1},\dots,\bs{\xi}_{n,s-1},\bs{\xi}_{n,s},\bs{\zeta}_{n,s},\bs{\zeta}_{n,s-1},\dots,\bs{\zeta}_{n,1},\bs{\zeta}_{n,0} = \bf{y} \\ \hbox{is a $\delta$-path which stays within}\ \O(\calO,\s_n)\Big\}
\end{multline*}
has positive probability under $\bbQ$: so, in particular, such a $\delta$-path exists.  Since $\delta$ was arbitrary, this completes the proof.
\end{proof}

\begin{rmk}
Since the choice of $s$ in the above proof depends only on $\delta$, it actually shows that Bernoulli shifts have uniformly connected model spaces rel $\S$, as in the remark at the end of Subsection~\ref{subs:connected}. \fin
\end{rmk}

\section{Analysis of Popa factors}

\subsection{Actions and cocycles for property-(T) groups}

Let $G$ be a finitely generated group. Let $S$ be a finite and symmetric generating set, and let $R$ be the set of all the corresponding relations in the free group on $S$ (including concatenations or conjugates of other relations).  Recall that $G$ has Kazhdan's \textbf{property (T)} if there is a $c > 0$ for which the following holds: whenever $\pi:G\actson V$ is a unitary representation, if there is some $v \in V$ such that $\|v\| = 1$ and
\begin{eqnarray}\label{eq:T}
\max_{s \in S}\|\pi^sv-v\| \leq c,
\end{eqnarray}
then $\pi$ has a nontrivial invariant vector.  The value of $c$ depends on the choice of $S$, but its existence does not.  See, for instance,~\cite{BekdelaHVal08}.

Now suppose that $(X,\mu,T)$ is a $G$-system and that $K\leq \bbT$ is a closed subgroup.  Let $\cal{U}(\mu,K)$ be the set of measurable functions $X\to K$ modulo agreement $\mu$-a.e.  This is a group under pointwise addition, and is naturally equipped with the topology of convergence in probability.  That topology is Polish, with a suitable metric given by the group-norm
\[\|f\|_\mu := \int |f|\,\d\mu,\]
where $|\cdot|$ is the quotient group-norm on $\bbT$ of the usual absolute value on $\bbR$.  The action $T:G \curvearrowright (X,\mu)$ induces an action of $G$ on $\cal{U}(\mu,K)$. We will need to work with the cohomology of this action in degree $1$, which is conveniently expressed in terms of the generators $S$ and relations $R$.

Firstly, a \textbf{$K$-valued $1$-cochain} is an equivalence class modulo $\mu$ of measurable functions $\a:S\times X\to K$, or equivalently an element of $\U(\mu,K)^S$.  For a $1$-cochain $\a$, we set
\[\|\a\|_{\mu,S} := \sum_{s\in S}\|\a(s,\cdot)\|_\mu.\]
This defines a Polish group-norm on the group of $1$-cochains under pointwise addition, similarly to $\U(\mu,K)$.

Next, given a $1$-cochain $\a:S\times X\to K$ and a word $w = s_\ell s_{\ell-1}\dots s_1$ over the alphabet $S$, define
\[\a(w,\cdot) := \a(s_1,\cdot) + \a(s_2,T^{s_1}(\cdot)) + \cdots + \a(s_\ell,T^{s_1\dots s_{\ell-1}}(\cdot)).\]
Restricting to relations, the resulting function
\[R\times X\to K:(w,x)\mapsto \a(w,x)\]
is the \textbf{$2$-coboundary} of $\a$, denoted by $d\a$. For any $\g:R\times X\to \bbT$ and finite $F \subseteq R$, we set
\[\|\g\|_{\mu,F} := \max_{w \in F}\|\g(w,\cdot)\|_\mu.\]

A $1$-cochain is a \textbf{$1$-cocycle} if $d\a = 0$ a.s. These form a closed subgroup $Z^1(T,\mu,K)$ of the group of $1$-cochains.

Lastly, if $\b \in \cal{U}(\mu,K)$, then its \textbf{$1$-coboundary} is the $1$-cocycle
\[d\b:(s,x) \mapsto \b(T^sx) - \b(x).\]
These form a further subgroup $B^1(T,\mu,K) \leq Z^1(T,\mu,K)$, not necessarily closed, and the quotient of these groups is the \textbf{first cohomology group} $H^1(T,\mu,K)$.

The next result is due to Schmidt~\cite[Theorem 3.4]{Sch81} and independently to Zimmer~\cite[Theorem 2.11]{Zim81}.  We include a proof in order to show that the relevant constants do not depend on the action $T$, or on the choice of the closed subgroup $K$ of $\bbT$.

\begin{thm}\label{thm:Sch}
If $G$ has property (T), then there is some $r > 0$ with the following property.  Let $(X,\mu,T)$ be an ergodic $G$-system.  For any closed subgroup $K \leq \bbT$ and any $\a \in Z^1(T,\mu,K)$, we have
\[\|\a\|_{\mu,S} \leq r \quad \Longrightarrow \quad \a \in B^1(T,\mu,K).\]
\end{thm}

\begin{proof}
\emph{Step 1.}\quad First suppose that $K = \bbT$.

Let $c$ be the constant in the definition of property (T), let $r := c^2/4\pi$, and consider $\a \in Z^1(T,\mu,\bbT)$ with $\|\a\|_{\mu,S} \leq r$.  Let $\pi:G\actson L^2(\mu)$ be the Koopman representation of $T$ twisted by $\a$:
\[(\pi^gf)(x) := \rme^{2\pi\rmi \a(g^{-1},x)}\cdot f(T^{g^{-1}}x).\]
This is a well-defined $G$-action because $\a$ satisfies the defining equations for a $1$-cocycle.

Now observe that
\[\|\pi^s1_X - 1_X\|_{L^2(\mu)} = \sqrt{\int |\rme^{2\pi\rmi \a(s^{-1},x)} - 1|^2\,\mu(\d x)} \leq \sqrt{4\pi r} = c,\]
where $1_X$ is the constant function $1$ on $X$.  Therefore property (T) gives some $f \in L^2(\mu)$ such that $\|f\|_{L^2(\mu)} = 1$ and $\pi^gf = f$ for all $g \in G$.

This fixed-point equation implies that $|f|$ is $T$-invariant, hence $\mu$-a.s. constant by ergodicity.  Therefore $f$ is actually $\rmS^1$-valued, and the fixed-point equation reads
\[\rme^{2\pi\rmi \a(g,x)} = \ol{f(T^gx)}f(x) \quad \forall g \in G.\]
Taking arguments, this asserts that $\a \in B^1(T,\mu,\bbT)$, as required.

\vspace{7pt}

\emph{Step 2.}\quad Now consider general $K \leq \bbT$.  If we regard $\a$ as a $\bbT$-valued function, Step 1 gives some $\b_0 \in \cal{U}(\mu,\bbT)$ such that
\[\a(g,x) = \b_0(T^gx) - \b_0(x) \quad \hbox{for}\ \mu\hbox{-a.e.}\ x\ \forall g \in G.\]
Since $\a$ is actually $K$-valued, this implies that the coset $\b_0(x) + K$ is $T$-invariant, and hence a.s. constant by ergodicity.  Let $c + K$ be that coset, and let $\b(x):= \b_0(x) - c$.  Then $\b$  takes values in $K$ almost surely, and still satisfies
\[\a(g,x) = \b(T^gx)-\b(x) \quad \hbox{for}\ \mu\hbox{-a.e.}\ x\ \forall g \in G.\]
\end{proof}

Next we introduce a `roughened' notion of cocycles.  Given a $G$-system $(X,\mu,T)$, $\eps > 0$, and finite $F \subseteq R$, an element $\a \in \U(\mu,K)^S$ is a \textbf{$K$-valued $(F,\eps)$-near-cocycle over} $(X,\mu,T)$ if
\[\|d\a\|_{\mu,F} < \eps.\]
The set of these will be denoted $Z^1_{F,\eps}(T,\mu,K)$.  For these we have the following roughened version of Theorem~\ref{thm:Sch}.

\begin{thm}\label{thm:noisySch}
Suppose $G = \langle S\,|\,R\rangle$ has property (T), and let $r > 0$ be as Theorem~\ref{thm:Sch}.  Then for every $r' > 0$ there are $\eps > 0$ and finite $F \subseteq R$ such that the following holds.  If $(X,\mu,T)$ is an ergodic $G$-system, and $\a \in Z^1_{F,\eps}(T,\mu,K)$ satisfies
\[\|\a\|_{\mu,S} \leq r,\]
then there is some $\b \in \cal{U}(\mu,K)$ such that
\[\|\a - d\b\|_{\mu,S} < r'.\]
\end{thm}

\begin{proof}
\emph{Step 1.}\quad We first convert the result to an assertion about invariant measures on a fixed $G$-space.  Let $Y:= (K^S)^G = K^{S\times G}$, equipped with the $G$-action by coordinate right-shift, and define the canonical $1$-cochain $\a_0:S\times Y\to K$ by
\[\a_0(s,y):= y_{s,e}.\]
We will prove that the desired result is implied by the following:
\begin{quote}
\emph{For every $r' > 0$ there are $\eps > 0$ and finite $F \subseteq G$ such that the following holds.  If $\nu$ is an ergodic shift-invariant Borel probability on $Y$, and if
\[\|\a_0\|_{\nu,S} \leq r \quad \hbox{and} \quad \|d\a_0\|_{\nu,F} < \eps,\]
then there is some $\b_0 \in \cal{U}(\nu,K)$ such that
\[\|\a_0 - d\b_0\|_{\nu,S} < r'.\]}
\end{quote}

Indeed, suppose this result is known, and consider $(X,\mu,T)$ and $\a$ as in the statement of the theorem.  Define
\[\phi:X\to Y:x\mapsto (\a(s,T^gx))_{s \in S,g \in G}.\]
Then $\phi$ intertwines $T$ with the coordinate shift on $Y$, and so $\nu := \phi_\ast\mu$ is an ergodic $G$-invariant probability on $Y$.  Moreover, the definition of $\phi$ gives $\a = \a_0\circ \phi$, and so
\[\|\a_0\|_{\nu,S} = \|\a\|_{\mu,S} \leq r \quad \hbox{and} \quad \|d\a_0\|_{\nu,F} = \|d\a\|_{\mu,F} < \eps.\]
So the claim above provides some measurable $\b_0:Y\to K$ such that, setting $\b:= \b_0\circ\phi$, we have
\[\|\a - d\b\|_{\mu,S} = \|\a_0 - d\b_0\|_{\nu,S} < r'.\]

\vspace{7pt}

\emph{Step 2.}\quad The rest of the proof is by contradiction.  Fix an increasing sequence $(F_i)_{i\geq 1}$ of finite sets whose union is $R$, and suppose that one could find $r' > 0$ and a sequence $(\nu_i)_{i\geq 1}$ of ergodic shift-invariant Borel probabilities on $Y$ such that
\begin{eqnarray}\label{eq:ass-on-nu}
\|\a_0\|_{\nu_i,S} \leq r \quad \hbox{and} \quad \|d\a_0\|_{\nu_i,F_i} < 2^{-i}
\end{eqnarray}
for all $i$, but also such that
\begin{eqnarray}\label{eq:contradict}
\|\a_0 - d\b_0\|_{\nu_i,S} \geq r'.
\end{eqnarray}
for all measurable functions $\b_0:Y \to K$ and all $i$.

By passing to a subsequence, we may assume that these measures $\nu_i$ weak$^\ast$-converge to a measure $\nu$, which must still be shift-invariant.  Since $G$ has property (T), a theorem of Glasner and Weiss~\cite{GlaWei97} gives that the set of ergodic measures is weak$^\ast$-closed among the shift-invariant probability measures on $Y$, so this $\nu$ is still ergodic.

Now, since each $\a_0(s,\cdot)$ is continuous on $Y$, the two parts of~(\ref{eq:ass-on-nu}) give
\[\|\a_0\|_{\nu,S} = \lim_{i\to\infty}\|\a_0\|_{\nu_i,S} \leq r\]
and
\[\|d\a_0(w,\cdot)\|_{\nu} = \lim_{i\to \infty}\|d\a_0(w,\cdot)\|_{\nu_i} = 0 \quad \forall w \in \bigcup_i F_i = R.\]

Therefore $\a_0:S\times Y\to K$ is an element of $Z^1(\rm{shift},\nu,K)$ with $\|\a_0\|_{\nu,S} \leq r$, so Theorem~\ref{thm:Sch} gives some $\b_1 \in \cal{U}(\nu,K)$ such that
\[\a_0 = d\b_1 \quad \nu\hbox{-a.s.}.\]

Letting $\b_0$ be a continuous function that approximates $\b_1$ sufficiently well in $\nu$-probability, we can arrange that
\[\|\a_0 - d\b_0\|_{\nu,S} < r'.\]
Since $\b_0$ is continuous and $\nu_i \stackrel{\rm{weak}^\ast}{\to} \nu$, this implies that also
\[\|\a_0 - d\b_0\|_{\nu_i,S} < r'.\]
for all sufficiently large $i$.  This contradicts~(\ref{eq:contradict}).
\end{proof}


Since the bounds in Theorem~\ref{thm:noisySch} do not depend on the system $(X,\mu,T)$, they are nontrivial even for actions on finite sets.  In particular, suppose that $H < G$ is a finite-index subgroup, let $V := G/H$, and let $\G := (V,E)$ be the directed Schreier graph resulting from the generating set $S$.  Let $G$ act on $V$ by left-multiplication.  Observe that $S \times V$ is in canonical bijection with $E$, so a $1$-cochain for this system may be interpreted as a map $\a:E\to K$.  For any $F\subseteq R$, let $L_F$ be the set of based loops in $\G$ that correspond to walking around a relation from the set $F$, starting from any vertex.  We will identify such a loop by a pair $(v,w)$, where $v$ is its starting vertex and $w$ is the relation. For any $\a:E\to K$, we can now define $d\a:L_R\to K$ by
\[d\a(v,w) := \sum_{i=0}^{\ell-1}\a(s_i\cdots s_1 v,s_{i+1}s_i\cdots s_1 v),\]
where $w = s_\ell s_{\ell-1}\cdots s_1$. As before, a map $\a:E\to K$ is a $1$-cocycle if and only if $d\a = 0$: that is, it sums to zero around any loop in $\G$ which corresponds to a relation of the group.

Similarly to the case of general $G$-systems, let $Z^1(\G,K)$ be the set of $1$-cocycles $\a:E\to K$, let $B^1(\G,K) = d(K^V)$ be the subgroup of coboundaries, and for $\eps > 0$ and $F \subseteq R$ let $Z^1_{\eps,F}(\G,K)$ be the subset of $1$-cochains $\a$ which satisfy
\[\frac{1}{|V|}\sum_{v \in V}|d\a(v,w)| < \eps \quad \forall w\in F.\]

Let $d$ be the metric on $\bbT^S$ given by
\[d(\theta_1,\theta_2) := \sum_{s\in S}|\theta_{1,s} - \theta_{2,s}|.\]
We also write $d$ for the restriction of this metric to $K^S$ for any closed subgroup $K$ of $\bbT$.  Since $S\times V$ is canonically identified with $E$, the normalized Hamming sums $d^{(V)}$ may be regarded as metrics on $\bbT^E$.  Now the translation of Theorems~\ref{thm:Sch} and~\ref{thm:noisySch} into this setting implies the following.

\begin{cor}\label{cor:well-sptd-cosets}
Let $r$ be as in Theorem~\ref{thm:Sch}.

\begin{enumerate}
\item If $\a \in Z^1(\G,K)$ and $d^{(V)}(\a,B^1(\G,K)) \leq r$ then in fact $\a \in B^1(\G,K)$.

\item Given $r' > 0$, let $\eps > 0$ and $F\subseteq R$ be as in Theorem~\ref{thm:noisySch}. If $\a \in Z^1_{\eps,F}(\G,K)$ and $d^{(V)}(\a,B^1(\G,K)) \leq r$ then in fact
\[d^{(V)}(\a,B^1(\G,K)) < r'.\]
\end{enumerate}
\end{cor}

\begin{proof}
Suppose $d\b \in B^1(\G,K)$ is such that $d^{(V)}(\a,d\b) \leq r$.  It suffices to prove either part with $\a$ replaced by $\a - d\b$, and so we may simplify the assumption to $d^{(V)}(\a,0) \leq r$.  What remains is a special case of Theorem~\ref{thm:noisySch}.
\end{proof}

\begin{rmk}
The relation between $\eps$ and $r'$ in part 2 of this corollary is reminiscent of recent work on coboundary expansion for simplicial complexes by Kaufman, Kazhdan and Lubotzky~\cite{KauKazLub14}.  They prove a similar inequality for $K = \bbZ/2\bbZ$ and for a family of graphs $\G$ constructed from certain Bruhat-Tits buildings.  Their family of examples is much more specialized than ours, but they obtain a more quantitative result: a linear dependence of $\eps$ on $r'$. They need this extra precision to deduce some other expansion properties of these complexes.  It would be interesting to know whether one can be so precise under our more general conditions. \fin
\end{rmk}

\subsection{Proof of Theorem D}

Now let $G$ be an infinite, finitely generated, residually finite group having Kazhdan's property (T), and let $e$ be its identity element.  Let $S$ and $R$ be as in the previous subsection. Such a group $G$ has a descending sequence $G_1 > G_2 > \dots$ of finite-index subgroups such that $\bigcap_n G_n = \{e\}$.

For the proof of Theorem D we will need those subgroups to have two additional properties.

\begin{lem}\label{lem:nontrivAb}
There is a descending sequence $(G_n)_{n\geq 1}$ of finite-index subgroups as above such that
\begin{itemize}
	\item[(i)] they all have nontrivial Abelianization, and
	\item[(ii)] the left-multiplication actions $\s_n:G \to \rm{Sym}(G/G_n)$ form a sofic approximation to $G$.
\end{itemize}
\end{lem}


\begin{proof}
First, as is standard, there is a sequence $G > H_1 > H_2 > \dots$ of finite-index \emph{normal} subgroups converging to $\{e\}$.

For each $n$, consider the quotient homomorphism
\[q_n:H_n \onto H_n/H_{n+1}.\]
Since the target is a nontrivial group, it contains a nontrivial cyclic subgroup $K_n \leq H_n/H_{n+1}$.  Letting $G_n := q_n^{-1}(K_n)$, we obtain $H_n \geq G_n > H_{n+1}$.  Let $\s_n$ be the left-multiplication action of $G$ on $G/G_n$.

We can now deduce the required properties.
\begin{itemize}
	\item[(i)] By construction, $K_n$ is a quotient group of $G_n/[G_n,G_n]$, so this latter is nontrivial for every $n$.
	\item[(ii)] If left-multiplication by $g \in G$ has a fixed point in $G/G_n$, then it has a fixed point in the further quotient $G/H_n$.  Since $H_n$ is normal in $G$, this requires that $g \in H_n$.  Since the subgroups $H_n$ converge to $\{e\}$, it follows that, for any $g \neq e$, the permutation $\s_n^g$ has no fixed points for all sufficiently large $n$.  This implies that $(\s_n)_{n\geq 1}$ is a sofic approximation.
\end{itemize}
\end{proof}

Henceforth $G$, $S$ and the sequence $(G_n)_n$ given by the above lemma will be fixed.  Let $\G_n := (V_n,E_n)$ be the sequence of directed Schreier graphs on the vertex sets $V_n = G/G_n$ and with $E_n$ defined by the generating set $S$. These form a sofic approximation to $G$ via the homomorphisms $\s_n$ in part (ii) of the above lemma. In this setting we often write $g\cdot v$ in place of $\s_n^g(v)$ for $g \in G$ and $v \in V_n$.

Now let us return to the Popa factor map described in the Introduction.  An isomorphic factor map whose target is a shift-system may be constructed as follows:
\[\Phi:\bbT^G \to \bbT^{S\times G}:(\theta_g)_{g\in G} \mapsto (\theta_{sg} - \theta_g)_{s\in S,g\in G}.\]
This equals $\phi^G$ for the map
\[\phi:\bbT^G\to \bbT^S:\theta \mapsto (\theta_s - \theta_e)_{s \in S}.\]
Since $\phi$ depends on only finitely many coordinates in $\bbT^G$, it is an $\eta$-AL approximation to itself for every $\eta > 0$, and we may work directly with the maps $\phi^{\s_n}$ on model spaces.

This $\Phi$ is a homomorphism of compact Abelian groups.  Since $S$ generates $G$, the kernel of $\Phi$ is the diagonal subgroup of $\bbT^G$.  The image of $\Phi$ is a compact subgroup of $\bbT^{S\times G}$, which we denote by $Z$.  The image measure $\Phi_\ast m^{\times G}$ must equal the Haar measure $\mu_Z$ of $Z$.  Therefore $\Phi$ is a factor map
\[(\bbT^G,m^{\times G},S) \to (\bbT^{S\times G},\mu_Z,S)\]
which is equivalent to the Popa factor map from the Introduction.

To prove Theorem D, we will need a description of the spaces of good models for this factor system.  This begins with the following, which is an immediate consequence of the definitions.

\begin{lem}\label{lem:model-sp-for-quot}
For any finite $F\subseteq R$ and $\eps > 0$ there is a w$^\ast$-neighbourhood $\calO$ of $\mu_Z$ such that
\[\O(\calO,\s_n) \subseteq Z^1_{\eps,F}(\G_n,\bbT) \quad \forall n \geq 1.\] \qed
\end{lem}

If $E \subseteq G$ and $\pi_E:\bbT^{S\times G}\to \bbT^{S\times E}$ is the coordinate projection, then $\pi_E$ is a group homomorphism, and the image measure $(\mu_Z)_E = (\pi_E)_\ast\mu_Z$ equals the Haar measure $\mu_{\pi_E(Z)}$.

\begin{lem}\label{lem:groups-equal}
Let $E \subseteq G$ be finite, let $C \subseteq G$ be a finite subset which is connected in $\rm{Cay}(G,S)$ and such that $C \supseteq E\cup SE$, and let $n$ be so large that, for any $v\in V_n$, the map
\[C\to V_n:g\mapsto g\cdot v\]
is a graph isomorphism between the restriction of $\rm{Cay}(G,S)$ to $C$ and the restriction of $\G_n$ to $C\cdot v$.  Then for any $\a \in Z^1(\G_n,\bbT)$ and any $v \in V_n$ we have
\begin{equation}\label{eq:groups-equal}
\big\{(\g(s,g\cdot v))_{s\in S,g \in E}:\ \g \in \a + B^1(\G_n,\bbT)\big\} = \pi_E(Z).
\end{equation}
\end{lem}

\begin{proof}
We first reduce to the case $\a = 0$.  To this end, observe that if $\a:E_n\to \bbT$ is a cocycle, then it has zero sum around any based loop in $\G_n$ which corresponds to a relation of the group.  The map
\begin{equation}\label{eq:g-dot-v}
C \to C\cdot v:g\mapsto g\cdot v
\end{equation}
defines an isomorphism of the restricted graphs by assumption, so any loop of $\G_n$ that is contained in $C \cdot v$ must arise from a relation of the group.  Therefore the cocycle condition implies that $\a$ sums to zero around any loop contained in $C \cdot v$, and hence
\[(\a|_{S\times g\cdot v})_{g \in E} = ((d\b)|_{S\times g\cdot v})_{g \in E}\]
for some $\b \in \bbT^{C\cdot v}$ obtained by simply summing along paths from some distinguished basepoint in $C$: all of $C$ can be reached this way because $C$ is connected.  Extending $\b$ arbitrarily to a member of $\bbT^{V_n}$, this shows that the left-hand side of~(\ref{eq:groups-equal}) does not depend on $\a$: for every $\a$ that left-hand side is a coset of a homomorphic image of $B^1(\G_n,\bbT)$, and we have just seen that any two of these cosets overlap, hence are equal.

Now suppose that $\a= 0$.  We will prove two separate inclusions.

First, if $\theta \in \bbT^G$ and
\[\g = \pi_E(\Phi(\theta)) = (\theta_{sg} - \theta_g)_{s \in S,g\in E} \in \pi_E(Z),\]
then $\g$ is also equal to $((d\b)|_{S\times g\cdot v})_{g \in E}$ for any choice of $\b \in \bbT^{V_n}$ satisfying
\[\b_{g\cdot v} = \theta_g \quad \forall g \in E.\]
Such a choice exists because $C \supseteq E$ and the map~(\ref{eq:g-dot-v}) is injective.

On the other hand, if $\b \in \bbT^{V_n}$, then the reverse of this argument produces some $\theta \in \bbT^G$ such that $\b_{g\cdot v} = \theta_g$ for all $g \in E$.
\end{proof}

\begin{lem}\label{lem:model-sp-for-quot2}
For any w$^\ast$-neighbourhood $\calO$ of $\mu_Z$, we have
\[\inf_{\a \in Z^1(\G_n,\bbT)}((\phi^{\s_n})_\ast m^{\times V_n})\big(\O(\calO,\s_n) - \a\big) \to 1\]
as $n\to\infty$.
\end{lem}

\begin{proof} For $\a \in Z^1(\G_n,\bbT)$, let
\[R_\a:\bbT^{V_n}\to \bbT^{V_n}:\b \mapsto \b + \a\]
be the corresponding rotation.  We need to show that the image measures
\[(R_{\a_n})_\ast (\phi^{\s_n})_\ast m^{\times V_n}\]
are asymptotically supported on $\O(\calO,\s_n)$ as $n\to\infty$ for any sequence $\a_n \in Z^1(\G_n,\bbT)$.  We will deduce this by showing that
\[(R_{\a_n})_\ast (\phi^{\s_n})_\ast m^{\times V_n} \q \mu_Z,\]
where this refers to quenched convergence as in~\cite[Definition 5.3]{Aus--soficentadd}.

Since $(\bbT^{S\times G},\mu_Z,S)$ is a factor of a Bernoulli shift, it is ergodic.  Therefore by~\cite[Corollary 5.7]{Aus--soficentadd} quenched convergence will follow if we show that 
\[(R_{\a_n})_\ast (\phi^{\s_n})_\ast m^{\times V_n} \lws \mu_Z \quad \hbox{(local weak$^\ast$ convergence)}.\]
The measure $(R_{\a_n})_\ast (\phi^{\s_n})_\ast m^{\times V_n}$ is equal to the Haar measure on the coset
\[B^1(\G_n,\bbT) + \a_n.\]
The result now follows from Lemma~\ref{lem:groups-equal}.  According to that lemma, for any finite $E \subseteq G$, the projection of this coset to the directed edges which emanate from $E\cdot v$ is simply a copy of $\pi_E(Z)$, provided $n$ is large enough depending on $E$.  For such $n$, the projection of $(R_{\a_n})_\ast (\phi^{\s_n})_\ast m^{\times V_n}$ to $E\cdot v$ is therefore equal to $\mu_{\pi_E(Z)} = (\mu_Z)_E$.
\end{proof}

\begin{proof}[Proof of Theorem D]
Let $r$ be as in Theorem~\ref{thm:Sch}, let $\delta := r/4$, and let $\eps > 0$ and $F\subseteq R$ be given by part 2 of Corollary~\ref{cor:well-sptd-cosets} for $r' := r/10$.  Finally, let $\calO$ be a w$^\ast$-neighbourhood of $\mu_Z$ as given by Lemma~\ref{lem:model-sp-for-quot} for this $F$ and $\eps$.

For any w$^\ast$-neighbourhood $\cal{U}$ of $\mu_Z$, Lemma~\ref{lem:model-sp-for-quot2} gives that
\begin{eqnarray}\label{eq:intersect-coset}
\O(\cal{U},\s_n) \cap (\a + B^1(\G_n,\bbT)) \neq \emptyset
\end{eqnarray}
for any $\a \in Z^1(\G_n,\bbT)$, once $n$ is sufficiently large.

Next, for each $n$, we may identify $G\actson \bbT^{G/G_n}$ as the induction to $G$ of the trivial action of $G_n$ on $\bbT$.  Hence Shapiro's Lemma~\cite[Theorem II.3.7]{Lang--gpcohom} gives
\[H^1(G,\bbT^{G/G_n}) \cong H^1(G_n,\bbT) = \rm{Hom}(G_n,\bbT) = \big(G_n/[G_n,G_n]\big)^\wedge \neq 0,\]
where the last conclusion is the point at which we use part (i) of Lemma~\ref{lem:nontrivAb}.  Therefore for each $n$ there is more than one coset of $B^1(\G_n,\bbT)$ in $Z^1(\G_n,\bbT)$.

Combining this with~(\ref{eq:intersect-coset}), it follows that we may choose sequences
\[\a_n \in B^1(\G_n,\bbT)\ \quad \hbox{and} \quad \a'_n \in Z^1(\G_n,\bbT)\setminus B^1(\G_n,\bbT)\]
satisfying
\[P^{\s_n}_{\a_n},P^{\s_n}_{\a'_n} \stackrel{\rm{weak}^\ast}{\to} \mu_Z \quad \hbox{as}\ n\to\infty.\]
By part 1 of Corollary~\ref{cor:well-sptd-cosets}, we must have that
\[d^{(V_n)}(\a_n',B^1(\G_n,\bbT)) \geq r\]
for all sufficiently large $n$.

Now suppose, for the sake of contradiction, that there were $\delta$-paths
\[\a_n = \a_{n,1},\a_{n,2},\ldots,\a_{n,\ell_n} = \a'_n\]
contained in $\O(\calO,\s_n)$ for arbitrarily large $n$.  By Lemma~\ref{lem:model-sp-for-quot}, they would also be contained in $Z^1_{\eps,F}(\G_n,\bbT)$.  Since $\delta = r/4$, there would have to be some $k \in \{2,\ldots,\ell_n-1\}$ for which
\[r' < r/4 < d^{(V_n)}(\a_{n,k},B^1(\G_n,\bbT)) < 3r/4.\]
Since $\a_{n,k} \in Z^1_{\eps,F}(\G_n,\bbT)$, this would contradict part 2 of Corollary~\ref{cor:well-sptd-cosets}.
\end{proof}

\subsection{Complemented factors}

The deduction of Corollary D$^{\prime\prime}$ rests on the following theorem about a general metric $G$-process $(\X^G,\mu,S,d_\X)$, which may be of independent interest.  It is expressed in terms of a sequence of measures on the spaces $\X^{V_n}$ which `doubly-quenched converge' to $\mu \in \Pr^S(\X^G)$ over $\S$, denoted by $\mu_n \dq \mu$.  This property is introduced in~\cite[Subsection 5.2]{Aus--soficentadd}. By~\cite[Corollary 5.18]{Aus--soficentadd}, the existence of such a sequence $\mu_n$ is an isomorphism-invariant of the process which does not depend on the metric $d_\X$.  We do not recall the definition here, but refer the reader to that reference for full details.

\begin{thm}\label{thm:comp-fact-model-surj}
Suppose there exist $\mu_n \in \Pr(\X^{V_n})$ such that $\mu_n \dq \mu$ rel $\S$.  Then any complemented factor map of $(\X^G,\mu,S)$ is model-surjective rel $\S$.

In particular, any complemented factor map of a Bernoulli shift is model-surjective relative to any sofic approximation of $G$.
\end{thm}

\begin{proof}
Suppose that $\Phi:(\X^G,\mu,S,d_\X) \to (\Y^G,\nu,S,d_\Y)$ is a complemented factor map, so there is another factor map $\Psi:(\X^G,\mu,S,d_\X) \to (\Z^G,\theta,S,d_\Z)$ such that the combined map
\[(\Phi,\Psi):(\X^G,\mu,S) \to (\Y^G\times \Z^G,\nu\times \theta,S)\]
is a measure-theoretic isomorphism.

By~\cite[Proposition 5.16]{Aus--soficentadd}, there are also measures $\theta_n \in \Pr(\Z^{V_n})$ such that $\theta_n \dq \theta$ rel $\S$: they can be obtained as images of the measures $\mu_n$.  Using these,~\cite[Corollary 5.13]{Aus--soficentadd} gives that the coordinate-projection factor map
\[\Pi:(\Y^G \times \Z^G,\nu\times \theta,S) \to (\Y^G,\nu,S)\]
is model-surjective rel $\S$. This conclusion is written in terms of w$^\ast$-neighbourhoods in~\cite{Aus--soficentadd}, but it is easily turned into condition 2 of Lemma~\ref{lem:model-surj-equiv}. Therefore the factor map
\[\Phi = \Pi\circ (\Phi,\Psi)\]
is a composition of a factor map which is model-surjective rel $\S$ and an isomorphism, so $\Phi$ itself is also model-surjective rel $\S$ by Propositions~\ref{prop:compose-still-surj} and~\ref{prop:isos-surj}.

This argument applies whenever $(\X^G,\mu,S)$ is a Bernoulli shift and $\S$ is a sofic approximation, because then the corresponding product measures on the spaces $\X^{V_n}$ always doubly-quenched converge to $\mu$~\cite[Lemma 5.11]{Aus--soficentadd}.
\end{proof}

\begin{proof}[Proof of Corollary D$^{\prime\prime}$]
This follows at once from the combination of Theorems A, C, D and~\ref{thm:comp-fact-model-surj}.
\end{proof}

%
%

\begin{rmk}
It is easy to prove that the Popa factor map itself is not complemented for any infinite discrete group $G$.  This is because it is a relatively compact extension, so if it had a complementing factor $\Xi:(\bbT^G,m^{\times G},S)\to (Z,\theta,R)$ then $(Z,\theta,R)$ would have to be a compact system, but the original Bernoulli shift $(\bbT^G,m^{\times G},S)$ has no compact factors.

Corollary D$^{\prime\prime}$ is much stronger: it prohibits the Popa factor system from appearing as a complemented factor of any other Bernoulli system in any way.  This holds only for certain special groups $G$.  By contrast, for amenable $G$ the Popa factor is always isomorphic to another copy of the Bernoulli shift by the general theory of~\cite{OrnWei87}, and the same holds if $G$ is a free group by an easy calculation. \fin
\end{rmk}

\subsubsection*{Acknowledgements}

This research was supported partly by a fellowship from the Clay Mathematics Institute and partly by the Simons Collaboration on Algorithms and Geometry

I am grateful for several helpful discussions with Lewis Bowen, David Fisher, Alex Lubotzky and Brandon Seward, and to Brandon Seward and the anonymous referee for correcting several errors in an earlier version. \fin

\bibliographystyle{abbrv}
\bibliography{bibfile}

\parskip 0pt
\parindent 0pt

\vspace{7pt}

%
%
%

\small{Courant Institute, New York University}

\small{251 Mercer St, New York, NY 10012, U.S.A.}

\small{\texttt{tim@cims.nyu.edu}}

\end{document}